  \documentclass[12pt]{amsart}
\usepackage[utf8]{inputenc}
\usepackage[T2A]{fontenc}
\usepackage{amssymb, amsmath, amsthm}
\usepackage{amscd}
\usepackage{graphicx}
\usepackage{dutchcal}
\usepackage{multicol}

\usepackage[backref=page]{hyperref}
\hypersetup{
    colorlinks=true,
    linkcolor=blue,
    filecolor=cyan,      
    urlcolor=cyan,
}
\renewcommand*{\backrefalt}[4]{%
    \ifcase #1 (Not cited.)%
    \or        (Cited on page~#2.)%
    \else      (Cited on pages~#2.)%
    \fi}
 
\usepackage[a4paper,hmargin=2.5cm,vmargin=2.5cm]{geometry}
\usepackage{marginnote}
\usepackage[english]{babel}
\graphicspath{ {images/} }
\usepackage{wrapfig}
\usepackage[matrix,arrow,curve]{xy}\usepackage{mathabx}
\usepackage[OT2, T1]{fontenc}

\DeclareSymbolFont{cyrletters}{OT2}{wncyr}{m}{n}
\DeclareMathSymbol{\Sha}{\mathalpha}{cyrletters}{"58}

\usepackage{color}

\begin{document}

\newtheorem{prop}{Proposition}[section]
\newtheorem{thrm}[prop]{Theorem}
\newtheorem{lemma}[prop]{Lemma}
\newtheorem{cor}[prop]{Corollary}

\newtheorem{theorem}{Theorem}
\renewcommand*{\thetheorem}{\Alph{theorem}}
\newtheorem*{definition}{Definition}

\theoremstyle{definition}
\newtheorem{df}[prop]{Definition}
\newtheorem{ex}[prop]{Example}
\newtheorem{rmk}[prop]{Remark}
\newtheorem{conj}{Conjecture}
\newtheorem{cl}[prop]{Claim}

\renewcommand{\proofname}{\textnormal{\textbf{Proof:  }}}
\renewcommand{\C}{\mathbb C}
\renewcommand{\refname}{Bibliography}
\newcommand{\Z}{\mathbb Z}
\newcommand{\Q}{\mathbb Q}
\renewcommand{\o}{\otimes}
\newcommand{\R}{\mathbb R}
\renewcommand{\O}{\mathcal O}
\newcommand{\g}{\mathfrak g}
\newcommand{\D}{\mathcal D}
\renewcommand{\i}{\sqrt{-1}}
\newcommand{\di}{\partial}
\newcommand{\dibar}{\overline{\partial}}
\renewcommand{\Im}{\operatorname{Im}}
\renewcommand{\ker}{\operatorname{ker}}
\newcommand{\Hom}{\operatorname{Hom}}
\newcommand{\Alb}{\operatorname{Alb}}
\newcommand{\Hilb}{\operatorname{Hilb}}
\newcommand{\CP}{\mathbb{P}_{\mathbb{C}}}
\newcommand{\Isom}{\operatorname{Isom}}
\newcommand{\acts}{\lefttorightarrow}
\newcommand{\codim}{\operatorname{codim}}
\newcommand{\rk}{\operatorname{rk}}
\newcommand{\hdot}{{\:\raisebox{3pt}{\text{\circle*{1.5}}}}}
\newcommand{\Tw}{\operatorname{Tw}}
\newcommand{\tw}{\mathrm{tw}}
\renewcommand{\phi}{\varphi}
\newcommand{\tSha}{\widetilde{\Sha}}
\binoppenalty = 10000
\relpenalty = 10000

\title{Shafarevich--Tate groups of holomorphic Lagrangian fibrations}

\author{Anna Abasheva, Vasily Rogov}

\begin{abstract}
Consider a Lagrangian fibration $\pi\colon X\to \mathbb P^n$ on a hyperk\"ahler manifold $X$. There are two ways to construct a holomorphic family of deformations of $\pi$ over $\C$. The first one is known under the name {\em Shafarevich--Tate family} while the second one is the {\em degenerate twistor family} constructed by Verbitsky. We show that both families coincide. We prove that for a very general $X$ all members of the Shafarevich--Tate family are K\"ahler. There is a related notion of the {\em Shafarevich--Tate group} $\Sha$ associated to a Lagrangian fibration. The connected component of unity of $\Sha$ can be shown to be isomorphic to $\C/\Lambda$ where $\Lambda$ is a finitely generated subgroup of $\C$ and $\C$ is thought of as the base of the Shafarevich--Tate family. We show that for a very general $X$, projective deformations in the Shafarevich--Tate family correspond to the torsion points in $\Sha^0$. A sufficient condition for a Lagrangian fibration $X$ to be projective is the existence of a holomorphic section. We find sufficient cohomological conditions for the existence of a deformation in the Shafarevich--Tate family that admits a section.

%We define the Shafarevich--Tate group of a Lagrangian fibration on a hyperk\"ahler manifold. This notion generalizes earlier constructions of Shafarevich--Tate groups for elliptic fibrations and certain Lagrangian fibrations on manifolds of $K3^{[n]}$-type. We use it to get a new interpretation of degenerate twistor deformations of Lagrangian fibrations and prove that under some generality assumptions all members of a degenerate twistor family are K\"ahler. We also get a characterization of algebraic members of degenerate twistor families in terms of the Shafarevich--Tate group. Finally, we apply the technique of Shafarevich--Tate groups to obtain new topological obstructions for the existence of  sections on Lagrangian fibrations.
\end{abstract}


\maketitle
\vspace*{-10mm}
\tableofcontents
\vspace*{-10mm}
\section{Introduction}
A compact connected K\"ahler manifold $X$ is called a \emph{hyperk\"ahler manifold} if it is simply connected and $H^0(X,\Omega^2_X)$ is generated by a holomorphic symplectic form $\sigma$. To a Lagrangian fibration $\pi\colon X\to \mathbb P^n$, we can associate a certain group called the {\em Shafarevich--Tate group}.%\footnote{Strictly speaking, such manifolds are called {\em irreducible hyperk\"ahler manifolds}. We omit the word "irreducible" as non-irreducible hyperk\"ahler manifolds do not arise in this paper}. 

\begin{definition}
The {\em Shafarevich--Tate group}\footnote{It is sometimes called {\em analytic Shafarevich--Tate group} in the literature in order to distinguish it from {\em arithmetic Shafarevich--Tate group}.} $\Sha$ of a Lagrangian fibration is the abelian group $H^1(\mathbb P^n,Aut^0_{X/\mathbb P^n})$ where $Aut^0_{X/\mathbb P^n}$ is the connected component of unity of the sheaf of vertical automorphisms of $X$ over $\mathbb P^n$.
\end{definition}

Choose an affine open cover $\mathbb P^n = \bigcup U_i$. Denote $U_i\cap U_j$ by $U_{ij}$. A class $s\in\Sha$ can be represented by a \v{C}ech cocycle with coefficients in $Aut^0_{X/\mathbb P^n}$. That is to say, for every pair of indices $i,j$ we are given an automorphism $s_{ij}$ of $\pi^{-1}(U_{ij})$ that commutes with $\pi$. Let us reglue the manifolds $\pi^{-1}(U_i)$ by the automorphisms $s_{ij}$. We obtain a complex manifold $X^s$ equipped with a holomorphic projection $\pi^s \colon X^s \to \mathbb P^n$. The fibers of $\pi^s$ are isomorphic to the fibers of $\pi \colon X \to \mathbb P^n$. This new manifold is called the {\em Shafarevich--Tate twist} of $X$ by $s\in \Sha$. Markman studied Shafarefich-Tate twists of Lagrangian fibrations on manifolds of $K3^{[n]}$-type in \cite{markman2014lagrangian}. Our work was notably influenced by his paper.

Define the sheaf $\Gamma$ of finitely generated abelian groups by the exact sequence
$$
0\to\Gamma\to\pi_*T_{X/\mathbb P^n}\to Aut^0_{X/\mathbb P^n}\to 0.
$$
This exact sequence induces the long exact sequence of cohomology groups. It can be shown to be right exact.
$$
H^1(\mathbb P^n,\Gamma) \to H^1(\pi_*T_{X/\mathbb P^n}) \to \Sha \to H^2(\mathbb P^n,\Gamma) \to 0.
$$
The holomorphic symplectic form $\sigma$ on $X$ induces an isomorphism $\pi_*T_{X/\mathbb P^n}\cong \Omega^1_{\mathbb P^n}$ (see Lemma \ref{tangent cotangent isomorphism} and \cite[Theorem 2.1.11]{abasheva2024shafarevich}). Thus, $H^1(\pi_*T_{X/\mathbb P^n}) \cong H^{1,1}(\mathbb P^n) = \C$. Let $\Sha^0$ be the image of the map $H^1(\pi_*T_{X/\mathbb P^n}) \to \Sha$. We conclude that $\Sha^0$ is isomorphic to a quotient of $\C$ by the finitely generated abelian group $\operatorname{im}H^1(\mathbb P^n,\Gamma)$.

In \cite{verbitsky2015degenerate} the author associates to a Lagrangian fibration $\pi\colon X\to \mathbb P^n$ a family of deformations over $\C$ called a {\em degenerate twistor family}. The construction goes as follows. Let $\alpha$ be a closed $(1,1)$-form on $\mathbb P^n$. There exists a unique complex structure $I_\alpha$ on $X$ such that the form $\omega + \pi^*\alpha$ is a holomorphic $2$-form on $(X,I_\alpha)$ \cite[Thm. 3.5]{verbitsky2015degenerate}. The manifold $(X,I_\alpha)$ is the {\em degenerate twistor deformation} of $X$ with respect to the form $\alpha$. Note that the construction of degenerate twistor deformations has differential geometric flavour while the construction of Shafarevich--Tate twists is of complex analytic nature. Nevertheless, these two constructions turn out to yield the same result.

\begin{theorem}[Theorem \ref{deg.tw=ShaT}]\label{ShT is deg tw intro}
Pick a class $s\in H^1(\pi_*T_{X/\mathbb P^n})$. Consider the twist $X^s$ of $\pi\colon X\to \mathbb P^n$ by the image of $s$ in $\Sha$. Let $\alpha$ be a closed $(1,1)$-form on $\mathbb P^n$ representing the same class in $H^{1,1}(\mathbb P^n)\cong H^1(\pi_*T_{X/\mathbb P^n})$ as $s$. Then the complex manifolds $X^s$ and $(X,I_\alpha)$ are isomorphic as fibrations over $\mathbb P^n$.
\end{theorem}

Are Shafarevich--Tate twists of a hyperk\"ahler manifold hyperk\"ahler themselves? We prove that indeed, they admit a holomorphic symplectic form (Subsection \ref{symplectic}) but the question of K\"ahlerness turns out to be trickier. We need to introduce a notion of {\em M-special} hyperk\"ahler manifolds which is motivated by \cite[Def. 1]{markman2014lagrangian}. A hyperk\"ahler manifold $X$ is called {\em M-special} if the subspace $H^{2,0}(X)+H^{0,2}(X)\subset H^2(X,\C)$ contains a rational class. A very general hyperk\"ahler manifold is not M-special.

\begin{theorem}[Theorems \ref{Kahlerness}, \ref{projectivity=torsion}]
\label{kahlerness theorem intro}
Let $\pi\colon X\to \mathbb P^n$ be a Lagrangian fibration on a not M-special projective hyperk\"ahler manifold $X$. Then every Shafarevich--Tate twist $X^s$ of $X$ by an element $s\in \Sha^0$ is a hyperk\"ahler manifold. Moreover, the set of $s\in\Sha^0$ such that the twist $X^s$ is a projective manifold forms a non-empty torsor over the group of torsion points of $\Sha^0$.
\end{theorem}

In order to prove Theorem \ref{kahlerness theorem intro}, we first study the sheaf $\Gamma$. It turns out that $\Gamma\o\Q$ is isomorphic to $R^1\pi_*\Q$ (Proposition \ref{Gamma is almost R^1pi_*Z}). The Leray spectral sequence helps us to compute some cohomology groups of $R^1\pi_*\Q$. We derive from that a description of the group $\Sha^0$ solely in terms of the Hodge structure on $H^2(X,\Z)$ (Corollary \ref{description of sha zero}). As we have already seen, the group $\Sha^0$ is isomorphic to $\C/\Lambda$ for a finitely generated abelian group $\Lambda$. It turns out that $X$ is not M-special if and only if $\Lambda$ is a dense subgroup of $\C$. We use that the set of twists of $X$ that are K\"ahler manifolds is open and $\Lambda$-invariant to conclude that all Shafarevich--Tate twists of $X$ are K\"ahler. The statement about projectivity follows by carefully applying Huybrechts' criterion \cite[Thm. 3.11]{huybrechts2003compact}.

Since we published this work as a preprint, the results of Theorem \ref{kahlerness theorem intro} were improved by the first author \cite{abasheva2024shafarevich} and independently in \cite{soldatenkov2024hermitian}. Their works show that Theorem \ref{kahlerness theorem intro} holds even without the technical assumption that $X$ is not $M$-special.

In Section \ref{obstruction} we study whether there exists a Shafarevich--Tate twist $X^s$ of $\pi\colon X\to \mathbb P^n$ that admits a holomorphic section. We show that the obstructions to the existence of such a twist lie in $H^2(\mathbb P^n,\Gamma) = \Sha/\Sha^0$ (Corollary \ref{main property of a}). Assume that the fibers of $\pi$ are reduced and irreducible. We show that the group $H^2(\mathbb P^n,\Gamma)\o\Q$ can be embedded into $H^3(X,\Q)$ (Corollary \ref{restriction is surjective}). This implies the following theorem.

\begin{theorem}[Theorem \ref{when a vanishes}]\label{sections intro}
Let $\pi \colon X \to \mathbb P^n$ be a Lagrangian fibration on a compact hyperk\"ahler manifold over a smooth base. Assume that the following holds:
\begin{itemize}
\item the fibers of $\pi$ are reduced and irreducible;
\item $H^3(X, \Q) =0$;
\item $H^2(\mathbb P^n, \Gamma)$ is torsion-free.
\end{itemize}
Then there exists a unique $s\in \Sha^0$ such that $\pi^s \colon X^s \to \mathbb P^n$ admits a holomorphic section. In particular, the fibration $\pi$ can be deformed to a fibration with a holomorphic section through the Shafarevich--Tate family.
\end{theorem}

At the core of the proof of Theorem \ref{sections intro} is a Hard Lefschetz-type result for fibers of Lagrangian fibrations (Corollary \ref{Lefschetz}).  The argument is inspired by the one used by Deligne to prove that the Leray spectral sequence of a smooth projective submersion degenerates on the second page \cite{deligne1968theoreme}.

\hfill

The paper is organized as follows. The notions of a Shafarevich--Tate group and a degenerate twistor deformation, as well as basic facts about Lagrangian deformations are recalled in Section \ref{prelim}. The main result of Section \ref{ShT-properties} is a refined version of Theorem \ref{ShT is deg tw intro}. We also prove that Shafarevich--Tate twists are holomorphic symplectic in Subsection \ref{symplectic} and study the period map for the family of Shafarevich--Tate twists in Subsection \ref{period map}. Section \ref{connected component} is concerned with the description of the group $\Sha^0$ in terms of topological invariants of $X$. Its results are used in Section \ref{applications} to prove Theorem \ref{kahlerness theorem intro}. In Section \ref{obstruction} we discuss obstructions to existence of sections and prove Theorem \ref{sections intro}. 

\hfill

\textbf{Acknowledgments.} We are deeply grateful to Dmitry Kaledin, Rodion D\'eev, and Andrey Soldatenkov for stimulating conversations and encouragement. A.A. thanks Giulia Sacc\`a and Claire Voisin for their interest and helpful remarks, and Raymond Cheng for reading a draft of this paper. V.R. thanks Olivier Benoist for his insights about the content of Section \ref{obstruction}. The work on this project started during our stay at the summer school ''Algebra \& Geometry--2021'' in Yaroslavl, Russia organized by the Laboratory of Algebraic Geometry of Higher School of Economics and Yaroslavl State Pedagogical University. We are grateful to the organizers for their hospitality and inspiring atmosphere. A.A. was supported in part by Simons Foundation.

%%%%%%%%%%%%%%%%%%%%%%%%%%%%%%%%%%
%%%%%%%%%%%%%%%%%%%%%%%%%%%%%%%%%%%
%%%%%%%%%%%%%%%%%%%%%%%%%%%%%%%%%%

\section{Preliminaries}\label{prelim}

\subsection{Lagrangian fibrations}\label{Geometry of Lagrangian fibrations}

Let $X$ be a hyperk\"ahler manifold. Fix a holomorphic symplectic form $\sigma$ on $X$. Consider a map $\pi\colon X\to B$ to a normal base $B$ of dimension half of that of $X$. A map $\pi \colon X \to B$ is called \textit{a Lagrangian fibration} if it is surjective with connected fibers and the restriction of $\sigma$ to every smooth fiber vanishes. We always assume that the base $B$ of a Lagrangian fibration $\pi$ is isomorphic to $\mathbb P^n$. It is the case in all known examples (see e.g. \cite[Subection 1.2]{huybrechts2022lagrangian}). We refer the reader to \cite{huybrechts2003compact} for  basics on hyperk\"ahler manifolds and Lagrangian fibrations.

Throughout this paper, we write $X$ for a compact hyperk\"ahler manifold, $\sigma$ for a holomorphic symplectic form on $X$ and $\pi\colon X\to B$ for a Lagrangian fibration. For  $b\in B$ we denote the fiber of $\pi$ over $b$ by $F_b$. The subset of the regular values of $\pi$ is denoted by $B^\circ$. This is a Zariski open subset of $B$ and $D:= B \,\setminus\, B^{\circ}$ is known to be a divisor \cite[Prop. 3.1]{hwang2009characteristic}.

Smooth fibers of a Lagrangian fibration $\pi\colon X\to B$ are complex tori \cite{markushevich1986integrable}. One can show that they are abelian varieties even if $X$ is not projective \cite{campana2006isotrivialite}.

Let $X$ be a hyperk\"ahler manifold. Then $H^2(X, \Z)$ carries a non-degenerate symmetric bilinear form $q \colon H^2(X, \Z) \o H^2(X, \Z) \to \Z$ called the {\em Beauville--Bogomolov--Fujiki form} \cite{beauville1983varietes,fujiki1987rham}. It satisfies the following properties:
\begin{itemize}
\item[(1)] The Hodge decomposition
\[
H^2(X, \C) = H^{2,0}(X) \oplus H^{1,1}(X) \oplus H^{0,2}(X)
\]
is orthogonal with respect to $q$. The restriction of $q$ to $H^{2,0}(X) \oplus H^{0,2}(X)$ is a positive Hermitian form and the restriction of $q$ to $H^{1,1}(X)$ has the signature $(1, h^{1,1}(X)-1)$;
\item[(2)] ({\em Fujiki formula}) There exists a non-zero constant $c_X$ such that
\begin{equation*}
\label{Fujiki}
q(\omega, \omega)^n = c_X \int_X \omega^{2n}
\end{equation*}
\end{itemize}

Let $\eta = \pi^*[H]\in H^2(X,\mathbb Z)$ be the pullback of the class of a hyperplane. It follows from the Fujiki formula that $q(\eta,\eta) = 0$.

\begin{prop}[{\cite[Section 2]{oguiso2009picard}}]
\label{restriction}
Let $\pi \colon X \to B$ be a Lagrangian fibration, and $\eta\in H^2(X,\mathbb Z)$ as above. Then the restriction map $H^2(X,\mathbb Q) \to H^2(F_b, \mathbb Q)$ has rank $1$ for every smooth fiber $F_b=\pi^{-1}(b)\subset X$. Moreover, the kernel of this map is precisely 
$$
\eta^{\perp}:= \{ v \in H^2(X, \mathbb Q) \ | q(v, \eta)=0\}.
$$
\end{prop}

\begin{cor}
\label{relative polarisation}
With the assumptions of Proposition \ref{restriction}, there is a K\"ahler class $[\omega]$ on $X$ that restricts to an integral class on every smooth fiber $F_b$.
\end{cor}
\begin{proof}
By Proposition \ref{restriction} the restriction map $H^2(X,\mathbb Q)\to H^2(F_b,\mathbb Q)$, for any $b\in B^\circ$, factors through the one-dimensional quotient $H^2(X,\mathbb Q)/\eta^\perp$. This vector space is generated by an integral class. Every K\"ahler class $[\omega]$ on $X$ restricts to a non-zero class on a fiber. Therefore, the class $[\omega]$ restricts to an integral class after appropriate scaling.
\end{proof}

The following result was proven by de Cataldo, Rapagnetta and Sacc\`a \cite[Lemma 2.3.1]{cataldo2021hodge} under additional assumptions and by the first author in the general case \cite[Theorem 2.1.11]{abasheva2024shafarevich}.

\begin{lemma}\label{tangent cotangent isomorphism}
 Let $\pi\colon X\to B$ be a Lagrangian fibration. Consider the isomorphism $\sigma\colon \Omega^1_X\xrightarrow{\sim} T_X$ induced by the holomorphic symplectic form $\sigma$. Then the map $\sigma$ sends $\pi^*\Omega^1_B\subset \Omega^1_X$ isomorphically to $T_{X/B}:=\ker({d\pi\colon T_X\to \pi^*T_B})$. In particular, the sheaves $\Omega^1_B$ and $\pi_{*}T_{X/B}$ are isomorphic and $\pi_*T_{X/B}$ is torsion-free.
\end{lemma}

Assume that $X$ is projective. Let $\omega$ be a closed $(1,1)$-form on $X$ that represents a rational class $[\omega] \in H^{1,1}_\Q(X)$. The contraction of $\omega$ with holomorphic vector fields defines a map
$$
    \tilde{\omega}\colon \pi_*T_{X/B}\to R^1\pi_*\O_X.
$$
We get the following maps by taking the exterior powers of $\tilde{\omega}$:
$$
    \widetilde{\omega}^i\colon \Omega^i_B\cong\Lambda^i(\pi_*T_{X/B})\to \Lambda^iR^1\pi_*\O_X \to R^i\pi_*\O_X, \quad i=1, \ldots, n.
$$

\begin{prop}
[{\cite[Thm. 1.2]{matsushita2005higher}}]
\label{Matsushita isomorphism}
Let $\pi \colon X \to B$ be a Lagrangian fibration on a  projective hyperk\"ahler manifold $X$ over a smooth base $B$. Then the map
\[
\widetilde{\omega}^i\colon \Omega^i_B\to R^i\pi_*\O_X.
\]
is an isomorphism for every $i=1,\dots,n$. In particular, the sheaves $R^i\pi_*\O_X$ are locally free. 
\end{prop}
\begin{rmk}
\label{not projective}
The assumption of projectivity of $X$ can be dropped, as was shown in \cite{soldatenkov2021moser}.

Let $[\omega]$ be a K\"ahler class on $X$ as in Corollary \ref{relative polarisation}. Then the restriction of the isomorphism $\pi_*T_{X/B}\to R^1\pi_*\mathcal O_X$ to $B^\circ$ is given by the contraction of $\omega$ with holomorphic vector fields \cite[Cor. 3.7]{soldatenkov2021moser}.
\end{rmk}

\begin{cor}\label{isomorphism of bases}
The following cohomology groups are isomorphic:
 $$H^1(B, \pi_*T_{X/B}) \simeq H^1(B, \Omega^1_B) \simeq H^1(B, R^1\pi_*\O_X) \simeq \C.$$
\end{cor}
\begin{proof}
The sheaves $\pi^*T_{X/B}$, $\Omega^1_B$, and $R^1\pi_*\mathcal O_X$ are isomorphic by Lemma \ref{tangent cotangent isomorphism} and Proposition \ref{Matsushita isomorphism}. The last isomorphism holds since $h^{1,1}(\mathbb P^n) =1$.
\end{proof}

%%%%%%%%%%%%%%%%%%%
%%%%%%%%%%%%%%%%%%%

\subsection{Shafarevich--Tate groups}\label{twists}

Let $\pi \colon X \to B$ be a surjective holomorphic map of complex manifolds. Consider the sheaf $Aut_{X/B}$ of \textit{vertical automorphisms} of $X$ over $B$. This is a sheaf of groups on $B$ defined as
$$
Aut_{X/B}(U) := \{\phi\colon \pi^{-1}(U)\:\tilde{\to}\: \pi^{-1}(U)\:|\:\pi\circ\phi = \pi\}.
$$
Let us define the subsheaf $Aut^0_{X/B}\subset Aut_{X/B}$ as follows. It consists of those vertical automorphisms $\phi$ whose restriction to every fiber $F_b$ lies in $Aut^0(F_b)$. Pick a class $s \in H^1(B, Aut^0_{X/B})$. We can represent it by a \v{C}ech cocycle $s_{ij}$ for an open cover $B = \bigcup U_i$. We are given an automorphism $s_{ij} \colon \pi^{-1}(U_{ij}) \to \pi^{-1}(U_{ij})$ for each pairwise intersection $U_{ij} := U_i \cap U_j$ . Let us glue a new complex manifold as
$$X^{s} := \bigsqcup \pi^{-1}(U_i) \Bigg / x\in\pi^{-1}(U_i) \sim s_{ij}(x)\in\pi^{-1}(U_j).$$
The manifold $X^s$ is equipped with a natural fibration $\pi^s\colon X^s\to B$. The isomorphism class of the fibration $\pi^s\colon X^s\to B$ depends only on the class of $s$ in $H^1(B,Aut^0_{X/B})$. The manifold $X^s$ is  called the {\it twist of $X$ by the class $s \in H^1(B, Aut^0_{X/B})$}. 

Suppose from now on that a general fiber of $\pi$ is a complex torus. Let $B^{\circ} \subset B$ be the set of values of $\pi$ for which the fiber $\pi^{-1}(b)$ is a complex torus.

\begin{lemma}
\label{commutativity}
Let $\pi\colon X\to B$ be as above. Then the sheaf $Aut^0_{X/B}$ is a sheaf of commutative groups.
\end{lemma}
\begin{proof}
The sheaf $Aut^0_{X/B}$ is a sheaf of commutative groups at least over $B^\circ\subset B$. Indeed, for any point $b\in B^\circ$ the group $Aut^0(F_b)$ is a complex torus, hence commutative. Let $g,h\in Aut^0_{X/B}(U)$ be two local sections of $Aut^0_{X/B}$. Then $[g,h]|_{B^\circ\cap U}$ is trivial. Therefore, $[g,h]$ is itself trivial, because $B^{\circ} \cap U$ is dense in $U$.
\end{proof}

\begin{df} [{\cite[Ch. I, Section 1.5.1]{friedman1994smooth}}]
\label{definition of Sha}
The group $H^1(B,Aut^0_{X/B})$ is called the {\it Shafarevich--Tate group} of the fibration $\pi\colon X\to B$ and is denoted by $\Sha$.
\end{df}

See \cite{dolgachev1994elliptic,kodaira1963compact,markman2014lagrangian} for applications of Shafarevich--Tate groups in study of elliptic and Lagrangian fibrations.

%%%%%%%%%%%%
%%%%%%%%%%%%%

\subsection{Degenerate twistor deformations}\label{parabolae explained}

The main references for this Subsection are \cite{verbitsky2015degenerate,bogomolov2022sections,soldatenkov2021moser}. Let $X$ be a hyperk\"ahler manifold. Fix a holomorphic symplectic form $\sigma \in H^{0}(X, \Omega^2_X)$. The key observation here is that the complex structure of $X$ is determined by the form $\sigma$, viewed as a smooth $\C$-valued $2$-form on $X$. Namely, the complex structure is uniquely determined by the subbundle of $(0,1)$-vectors $T^{0,1}X \subset TX$. This subbundle can be recovered as $\ker \sigma|_{TX\o\C}$.

We can characterize all smooth complex-valued $2$-forms on $X$ that can be realized as holomorphic symplectic forms for some complex structures. 

\begin{df}[{\cite{bogomolov2022sections}}]
Let $M$ be a smooth manifold of real dimension $4n$. Consider a smooth complex-valued $2$-form $\sigma \in \Gamma(\Lambda^2T^*M \o \C)$. It is called a \textit{c-symplectic structure} if the following hold:
\begin{itemize}
\item $d\sigma = 0$;
\item $\sigma^{n+1}=0$;
\item $\sigma^n \wedge \overline{\sigma}^n \neq 0$  at each point of $M$.
\end{itemize}
\end{df}

\begin{prop} [{\cite[Prop. 3.1]{verbitsky2015degenerate}}]
Let $\sigma$ be a c-symplectic structure on $M$. Define
$$T^{0,1}M := \ker \sigma \subset TM \o \C
$$ and $T^{1,0}M:= \overline{T^{0,1}M}$. Then the following holds:
\begin{itemize}
\item $TM = T^{0,1}M \oplus T^{1,0}M$;
\item $[T^{0,1}M, T^{0,1}M] \subseteq T^{0,1}M$.
\end{itemize}
Equivalently, the operator $I_{\sigma} \colon TM \to TM$ that acts as $\i$ on $T^{1,0}M$ and as $-\i$ on $T^{0,1}M$  is  an integrable complex structure on $M$. The form $\sigma$ is holomorphic symplectic with respect to $I_{\sigma}$. 
\end{prop}

\begin{lemma}[{\cite[Thm. 1.10]{verbitsky2015degenerate}}]
Let $M$ be a complex manifold (not necessarily compact) and $\sigma$ a holomorphic symplectic $2$-form on $M$. Consider a proper Lagrangian fibration $\pi \colon M \to S$ over a complex manifold $S$. Fix a closed $2$-form $\alpha$ on $S$ of Hodge type $(2,0)+(1,1)$. Then the form $\sigma_\alpha:= \sigma +\pi^*\alpha$ is a c-symplectic structure. Moreover, the projection $\pi \colon M \to S$ is holomorphic with respect to the complex structure $I_{\alpha}:= I_{\sigma_{\alpha}}$. It is a Lagrangian fibration with respect to the holomorphic symplectic form $\sigma_{\alpha}$. 
\end{lemma}

Suppose that the base $S$ of a Lagrangian fibration satisfies the condition $H^1(S,\mathcal O_S) = 0$. The next two statements will show that the isomorphism class of the complex manifold ($M$, $I_{\alpha}$) depends only on the cohomology class of $\alpha$.

\begin{lemma}[{\cite[Thm. 2.7]{soldatenkov2021moser}}]
\label{holomorphic moser}
Let $(M, \sigma)$ be a holomorphic symplectic manifold and $\pi \colon M \to S$ a proper Lagrangian fibration. Assume that $H^1(S,\mathcal O_S) = 0$. Let $\alpha$ be an exact $2$-form of type $(2,0)+(1,1)$ on $S$. Consider a family of c-symplectic structures
\[
\sigma_t := \sigma+t\pi^*\alpha.
\]
Then there exists a flow of diffeomorphisms $\phi_t$ on $M$ preserving the fibers of $\pi$ such that for each $t$ 
\[
\phi_t \colon (M, I_0) \to (M, I_t)
\]
is a biholomorphism.
\end{lemma}

\begin{cor}
Let $\pi\colon M\to S$ be as in Lemma \ref{holomorphic moser}. Consider two cohomologous $2$-forms $\alpha$ and $\alpha'$ of Hodge type $(2,0) + (1,1)$ on $S$. Let $I$ and $I'$ be the complex structures on $M$ induced by the c-symplectic forms $\sigma + \pi^*\alpha$ and $\sigma + \pi^*\alpha'$ respectively. Then there exists a biholomorphism $\phi\colon (M,I) \to (M,I')$ preserving the fibers of $\pi$.
\end{cor}

\begin{proof}
Follows by applying Lemma \ref{holomorphic moser} to the holomorphic symplectic manifold $(M,I)$ equipped with the holomorphic symplectic form $\sigma_{\alpha}:= \sigma + \pi^*\alpha$. Indeed, in this case the form $\sigma + \pi^*\alpha'$ can be written as $\sigma_{\alpha} + \pi^*d\beta$ for some $1$-form $\beta$.
\end{proof}

\begin{df}[{\cite[Def. 3.17]{verbitsky2015degenerate}}]\label{definition of dg.tw}
Let $\pi \colon X \to B$ be a Lagrangian fibration and $\alpha$ a closed $2$-form on $B$ of type $(2,0) + (1,1)$. Consider the family of c-symplectic structures $\sigma_t:= \sigma+t\pi^*\alpha$, $t\in\C$. Let $I_t$ be the associated family of complex structures on $X$. The \textit{degenerate twistor family} $\mathfrak{X}_{deg.tw}$ of $(X, \pi)$ is the following manifold. Its underlying smooth manifold is $X \times \C$ and the almost complex structure is defined as
\[
(I_{tw})_{(x,t)}:= I_t \oplus I_{\C},
\]
where $x\in X$, $t\in \C$, and $I_{\C}$ is the standard complex structure on $\C$.
\end{df}

One can show that the almost complex structure $I_{tw}$ from Definition \ref{definition of dg.tw} is integrable \cite[Thm. 3.18]{verbitsky2015degenerate}. In other words, $\mathfrak{X}_{deg.tw}$ is a complex manifold. The projection $\mathfrak{X}_{deg.tw} \to \C$ is holomorphic and  $\mathfrak{X}_{deg.tw}$ is the total space of a holomorphic family of complex holomorphically symplectic manifolds. It is endowed with a holomorphic fibration $\Pi_{deg.tw} \colon \mathfrak{X}_{deg.tw} \to B \times \C$ that restricts to a fiber $X_t \subset \mathfrak{X}_{deg.tw}$ as $\pi$.

%%%%%%%%%%%%%%%%%%%%%%%%%%%%%%%%
%%%%%%%%%%%%%%%%%%%%%%%%%%%%%%%
%%%%%%%%%%%%%%%%%%%%%%%%%%%%%%%

\section{Shafarevich--Tate group of a Lagrangian fibration}\label{ShT-properties}
\subsection{Structure of Shafarevich--Tate groups: first steps}\label{ShT-def}

Let $X$ be a compact hyperk\"ahler manifold and $\pi \colon X \to B$ a Lagrangian fibration, $B\simeq  \mathbb P^n$. Recall that the Shafarevich--Tate group $\Sha$ of $\pi$ is defined to be $H^1(B, Aut^0_{X/B})$ (Definition \ref{definition of Sha}). 

Consider the exponential map $\pi_*T_{X/B}\to Aut^0_{X/B}$. Define the sheaf $\Gamma$ by the short exact sequence
$$
    0 \to \Gamma \to \pi_*T_{X/B} \to Aut^0_{X/B}\to 0.
$$
It induces the following long exact sequence:
\begin{equation}
\label{exact seq of shas}
    H^1(B,\Gamma) \to H^1(B,\pi_*T_{X/B}) \to \Sha \to H^2(B,\Gamma) \to 0.
\end{equation}
Indeed, the sheaf $\pi_*T_{X/B}$ is isomorphic to $\Omega^1_B$ by Lemma \ref{tangent cotangent isomorphism}, hence, $$H^2(B,\pi_*T_{X/B}) = H^{1,2}(\mathbb P^n) = 0.$$

Write
\[
\tSha= \tSha(X, \pi):=H^1(B, \pi_*T_{X/B}).
\]
Of course, the group $\tSha(X, \pi)$ is (non-canonically) isomorphic to  $\mathbb C$ (Corollary \ref{isomorphism of bases}). However, we will use this notation when we want to emphasize its relation to $\Sha$.

\begin{lemma}\label{about Gamma}
The sheaf $\Gamma:=\ker (\pi_*T_{X/B}\to Aut^0_{X/B})$ is a sheaf of finitely generated torsion-free abelian groups.
\end{lemma}
\begin{proof}
The restriction of the sheaf $\Gamma$ to $B^\circ$ is a local system of torsion-free abelian groups of rank $n = \frac{1}{2}\dim X$. For every open $U\subset B$ the restriction map
\begin{equation}
\label{frf}
H^0(U,\Gamma)\to H^0(U\cap B^\circ,\Gamma)
\end{equation}
is injective. Indeed, an element of $H^0(U, \Gamma)$ is a vector field on $\pi^{-1}(U)$. If a vector field vanishes on an open subset of $\pi^{-1}(U)$, then it vanishes on $\pi^{-1}(U)$. The group on the right-hand side of (\ref{frf}) is torsion-free of finite rank, hence so is the group on the left-hand side.
\end{proof}

Lemma \ref{about Gamma} implies that the cohomology groups of $\Gamma$ are finitely generated abelian groups. Let $\Sha^0$ denote the image of the map $\tSha = H^1(B,\pi_*T_{X/B})\to \Sha$ from the long exact sequence (\ref{exact seq of shas}). The group $\Sha$ fits into the short exact sequence
\begin{equation}
\label{sha zero to sha}
    0 \to \Sha^0 \to \Sha \to H^2(B,\Gamma) \to 0.
\end{equation}

Being a quotient of $\tSha = \C$ by a subgroup, the group $\Sha^0$ inherits a structure of a connected topological group. Endow $\Sha$ with the translation invariant topology such that $\Sha^0$ is the connected component of unity of $\Sha$. We view $H^2(B,\Gamma)$ as a discrete topological group. Both maps in the exact sequence (\ref{sha zero to sha}) are continuous maps of topological groups.

The following provides useful intuition for Shafarevich--Tate groups.

\begin{rmk} Let $B$ be a complex manifold, $Y: = B \times \C^{\times}$ and $Aut^{0}_{Y/B} \subset Aut_{Y/B}$ the subsheaf  consisting of the automorphisms that lie in the connected component of unity. Then $Aut^0_{Y/B} = \O^{\times}_B$ and $H^1(B, Aut^0_{Y/B}) = \operatorname{Pic}(B)$. Total spaces of twists by classes in $H^1(B, Aut^{0}_{X/B})$ are holomorphic principal $\C^{\times}$-bundles over $X$. They are in one-to-one correspondence with holomorphic line bundles.

Thus, the Shafarevich--Tate group $\Sha$ serves as an analog of the Picard group, its subgroup $\Sha^0$ is an  analog of $\operatorname{Pic}^0(B)$, and $\Sha/\Sha^0$ of the N\'eron--Severi group $NS(B)$.

\end{rmk}

%%%%%%%%%%%%%%%
%%%%%%%%%%%%%%%

\subsection{Shafarevich--Tate family}\label{ShT-family}

Consider an element $s\in \tSha$. The twist $X^s$ of $X$ by $s$ is defined to be the manifold $X^{[s]}$  where $[s]$ is the image of $s$ under the map $\tSha\to \Sha$ (Subsection \ref{twists}). The twist $X^{s}$ is a complex manifold endowed with a holomorphic map $\pi^{s}:= \pi^{[s]}$ to $B$.

\begin{prop}
\label{prop_ShT_family}
There exists a holomorphic family of complex manifolds 
$$\mathfrak{X}_{\Sha \mathrm{T}} \to \tSha \cong \C$$
such that the fiber over $s \in \tSha$ is $X^s$. This family is endowed with a  holomorphic fibration 
$$
\Pi_{\Sha \mathrm{ T}} \colon \mathfrak{X}_{\Sha \mathrm{T}} \to B \times \tSha.
$$
and the map $\Pi_{\Sha \mathrm{T}}$ restricts to $\pi^s$ on each fiber $X^s$.
 \end{prop}
\begin{proof}
Set $\widetilde{B}:=B \times \tSha$ and $\widetilde{X} := X \times \tSha$. Let $\widetilde{\pi} := \pi \times id \colon \widetilde{X} \to \widetilde{B}$ be the projection. For any open covering $B = \bigcup_i U_i$  define $\widetilde{U_i}:= U_i \times \tSha$. Of course, $\widetilde B = \bigcup_i \widetilde{U_i}$ is an open covering of $\widetilde{B}$.

Let $v \in \tSha$ be a non-zero class. Represent it by a \v{C}ech cocycle $v_{ij}$ on some open covering $U_{ij}$. Let $\widetilde{v}_{ij}$ be the pullback of $v_{ij}$ to $\widetilde{U}_{ij}$. Define a $1$-cocycle on $\widetilde B$ with coefficients in $\widetilde{\pi}_*T_{\widetilde X/\widetilde B}$ as
\[
w_{ij} := t\widetilde{v_{ij}} \in \widetilde\pi_*T_{\widetilde{X}/\widetilde B}(\widetilde{U_{ij}})
\]
where $t$ is the coordinate on $\tSha \cong \C$. The twist of $\widetilde{\pi}\colon\widetilde{X} \to \widetilde{B}$ along the cocycle $\{\phi_{ij}\}:= \{\exp(w_{ij})\}$ is the desired family $\mathfrak{X}_{\Sha \mathrm{T}}$. In other words, the manifold $\mathfrak X_{\Sha\mathrm T}$ is obtained as
$$
\mathfrak{X}_{\Sha\mathrm{T}} = \bigsqcup \left(\pi^{-1}(U_i)\times\C\right) \Bigg/ \left(\pi^{-1}(U_i)\times\C\right) \ni x \sim 
\phi_{ij}(x)\in \left(\pi^{-1}(U_j)\times\C\right).
$$
\end{proof}

\begin{df} 
The family $\mathfrak{X}_{\Sha \mathrm{T}} \to \tSha$ constructed in Proposition \ref{prop_ShT_family} will be referred to as the \textit{Shafarevich--Tate family}.
\end{df}

%%%%%%%%%%%%%%%%%%%%%%
%%%%%%%%%%%%%%%%%%%%%%

\subsection{Shafarevich--Tate twists are symplectic}\label{symplectic}
 
A twist of a Lagrangian fibration by an element of $\Sha$ is a priori only a complex manifold. We will see in this Subsection that it is a holomorphic symplectic manifold.\footnote{A twist of a Lagrangian fibration might not be K\"ahler. That is why we use the term holomorphic symplectic instead of hyperk\"ahler.}

\begin{lemma}
\label{autos and symplectic forms}
Let $(M,\sigma)$ be a (not necessarily compact) holomorphic symplectic manifold and $\pi\colon M\to S$ a proper Lagrangian fibration on $M$ over a smooth base $S$. Consider an automorphism $\phi\in H^0(S, Aut^0_{M/S})$. Then
$$
\phi^*\sigma - \sigma = \pi^*\alpha
$$
where $\alpha$ is a closed holomorphic $2$-form on $S$.
\end{lemma}

\begin{proof}
The statement is local on $S$ so we can always shrink $S$ if necessary. For $S$ small enough, one can realize $\phi$ as $\exp(v)$ for a vertical holomorphic vector field $v$. There exists a holomorphic $1$-form $\beta$ on $S$ such that $\iota_v\sigma = \pi^*\beta$ (Lemma \ref{tangent cotangent isomorphism}). Let $\phi_t$ denote the flow of $v$ and $\mathsf L_v$ the Lie derivative along the vector field $v$. Then
$$
\phi^*\sigma - \sigma = \int\limits_0^1\frac{d(\phi_t^*\sigma)}{dt}dt = \int\limits_0^1 \phi_t^*\mathsf L_v\sigma dt = d\int_0^1\phi_t^*(\iota_v\sigma) dt = d\int\limits_0^1 \phi_t^*\pi^*\beta dt = \pi^*d\beta.
$$
The last identity holds because $\pi\circ \phi_t = \pi$. Thus $\phi^*\sigma - \sigma = \pi^*\alpha$ where $\alpha = d\beta$.
\end{proof}

\begin{thrm}
\label{Sha from symplectic automorphisms}
Let $\pi\colon X\to B$ be a Lagrangian fibration on a compact holomorphic symplectic manifold $X$. Denote by $Aut^{0,\sigma}_{X/B}$ the subsheaf of $Aut^0_{X/B}$ consisting of $\sigma$-symplectic automorphisms. Then the inclusion $Aut^{0,\sigma}_{X/B} \to Aut^0_{X/B}$ induces an isomorphism of cohomology groups $H^i(B, Aut^{0,\sigma}_{X/B}) \to H^i(B, Aut^0_{X/B})$ for $i=0,1$.
\end{thrm}

\begin{proof}
The isomorphism $H^0(B, Aut^{0,\sigma}_{X/B}) \xrightarrow{\sim} H^0(B, Aut^0_{X/B})$ follows from Lemma \ref{autos and symplectic forms} since $B=\mathbb{P}^n$ possesses no holomorphic $2$-forms.

We will  prove that the map $H^1(B,Aut^{0,\sigma}_{X/B}) \to H^1(B, Aut^0_{X/B})$ is surjective. The proof of injectivity of this map follows the same pattern and is left to the reader.

Consider a class $\phi \in H^1(B,Aut^0_{X/B})$. It can be represented by a \v{C}ech $1$-cocycle $\phi_{ij}\in Aut^0_{X/B}(U_{ij})$ for an open covering $B = \bigcup U_i$. Lemma \ref{autos and symplectic forms} implies that $\phi_{ij}^*\sigma - \sigma = \pi^*\alpha_{ij}$ for a closed holomorphic $2$-form $\alpha_{ij}$ on $U_{ij}$. 

The collection of forms $\{\alpha_{ij}\}$ defines a \v{C}ech $1$-cocycle on $B$ with coefficients in $\Omega^2_B$. As $H^1(B,\Omega^2_B)$ vanishes, the cocycle $\{\alpha_{ij}\}$ is exact. Consequently, there exist holomorphic $2$-forms $\beta_i$ on $U_i$ such that $\alpha_{ij} = \beta_j|_{U_{ij}} - \beta_i|_{U_{ij}}$. The form $\alpha_{ij}$ is closed for every $i, j$, hence $d\beta_i|_{U_{ij}} = d\beta_j|_{U_{ij}}$. That means that the forms $d\beta_i$ glue to a globally defined holomorphic $3$-form. There are no non-trivial holomorphic $3$-forms on $B = \mathbb P^n$, hence $d\beta_i = 0$ for every $i$. Therefore, the forms $\beta_i$ are exact. 

We can find a holomorphic $1$-form $\gamma_i$ on $U_i$ such that $\beta_i = d\gamma_i$. Let $v_i$ be the vertical vector field on $\pi^{-1}(U_i)$ such that $\iota_{v_i}\sigma = \pi^*\gamma_i$ and $\psi_i$ be the flow of $v_i$. Define $\widetilde{\phi_{ij}}$ to be the automorphism $$
\widetilde{\phi}_{ij}:= \phi_{ij}\circ\psi_i|_{U_{ij}}\circ\psi_j^{-1}|_{U_{ij}}.
$$ 
A direct computation shows that $\widetilde{\phi}_{ij}^*\sigma = \sigma$. At the same time $\{\widetilde{\phi}_{ij}\}$ define the same class in $H^1(B,Aut^0_{X/B})$ as $\{\phi_{ij}\}$. This proves the surjectivity of the map 
$$ H^1(B, Aut^{0,\sigma}_{X/B}) \to H^1(B, Aut^0_{X/B}).$$
\end{proof}

\begin{cor}
\label{sha tate are symplectic}
Let $\pi\colon X\to B $ be a Lagrangian fibration and $\pi^s\colon X^s\to B$ be the twist of $\pi$ by an element $s\in \Sha$. Then $X^s$ is a holomorphic symplectic manifold and $\pi^s$ is a Lagrangian fibration.
\end{cor}
\begin{proof}
Every element $s\in\Sha$ can be represented by a \v{C}ech cocycle $\phi_{ij}\in Aut^{0,\sigma}_{X/B}(U_{ij})$ (Theorem \ref{Sha from symplectic automorphisms}). The manifold $X^s$ is obtained by gluing the holomorphic symplectic manifolds $\pi^{-1}(U_i)$ by the automorphisms $\phi_{ij}$. Since the automorphisms $\phi_{ij}$  preserve $\sigma$, the forms $\sigma|_{\pi^{-1}(U_i)}$ are glued to a well-defined holomorphic symplectic form $\sigma^{s}$ on $X^{s}$. Locally $\pi^s$ coincides with $\pi$, therefore it is Lagrangian with respect to $\sigma^s$.
\end{proof}

%%%%%%%%%%%%%%%%
%%%%%%%%%%%%%%%%

\subsection{Degenerate twistor deformations as Shafarevich--Tate twists}\label{deg.tw. as ShT}

Let $\pi\colon X\to B$ be a Lagrangian fibration. Consider the Shafarevich--Tate family $\mathfrak{X}_{\Sha \mathrm{T}} \to \tSha$ (Proposition \ref{prop_ShT_family}) and the degenerate twistor family $\mathfrak{X}_{deg.tw} \to \C$ (Definition \ref{definition of dg.tw}) associated to $\pi$. Fix a holomorphic symplectic form $\sigma$ on $X$. It induces an isomorphism $\tSha\cong H^{1,1}(B)$. Let $[\alpha]\in H^{1,1}(B)$ denote the class of a hyperplane section of $B=\mathbb P^n$. The class $[\alpha]$ induces the natural isomorphism $H^{1,1}(B)\cong\C$.

\begin{thrm}\label{deg.tw=ShaT}
There exists an isomorphism of $\mathfrak{X}_{\Sha \mathrm{T}}$ and $\mathfrak{X}_{deg.tw}$ as families of complex manifolds that lifts the isomorphism $\tSha\cong \C$. In other words, we have a commutative diagram:
\[
\xymatrix{
\mathfrak{X}_{deg.tw} \ar[r]^{\Phi} \ar[d]_{\Pi_{deg.tw}} & \mathfrak{X}_{\Sha \mathrm T} \ar[d]^{\Pi_{\Sha \mathrm T}} \\
\C \ar[r]^{\simeq}&  \tSha.
}
\]
\end{thrm}

{\bf Proof. Step 1:} First, let us recall the relation between Dolbeault and \v{C}ech cohomology groups of the sheaf $\Omega^1_B$. The class $[\alpha]$ considered as a Dolbeault cohomology class, can be represented by a closed $(1,1)$-form $\alpha$ on $B$. Take an open cover $B = \bigcup U_i$. The class $[\alpha]$ may be represented by a \v{C}ech cocycle $\{a_{ij}\}$ where $a_{ij}\in H^0(U_{ij},\Omega^1_B)$ are holomorphic $1$-forms on $U_{ij}$. Given a closed $(1,1)$-form $\alpha$, one can construct $\{a_{ij}\}$ as follows. There exists a $(1,0)$-form $a_i$ on $U_i$ such that $\omega|_{U_i} = da_i = \dibar a_i$ for every $i$. We define the form $a_{ij}$ as 
\[
a_{ij}:= a_i|_{U_{ij}} - a_j|_{U_{ij}}.
\]
\hfill

{\bf Step 2:} In this step we will construct an isomorphism from $\mathfrak{X}_{\Sha\mathrm{T}}$ to $\mathfrak{X}_{deg.tw}$ locally on the base $B$. Consider the vertical vector fields $w_i$ on $\pi^{-1}(U_i)\times\C \subset X\times \C$ defined by the formula

$$
\iota_{w_i}\sigma =  t\pi^*a_i
$$
Here $t$ is the linear coordinate function on $\C$ and $\sigma$ is the pullback of the holomorphic symplectic form on $X$ to $\pi^{-1}(U_i) \times \C$. For any $t\in \C$ we have the following identity on $\pi^{-1}(U_i)\times\{t\}$:
$$
\mathsf L_{w_i} \sigma = t\pi^*da_i = t\pi^*\alpha|_{\pi^{-1}(U_i)}.
$$
Let $\phi_i$ be the exponential of the vector field $w_i$. This gives the following equality of forms on $\pi^{-1}(U_i) \times \{t\}$:
$$
    \phi_i^*\sigma = \sigma + t\pi^*\alpha,
$$
so that for any $t\in\C$ the map $\phi_i$ is a biholomorphism
$$
\phi_i|_{\pi^{-1}(U_i)\times\{t\}}\colon (\pi^{-1}(U_i),I_t) \to \pi^{-1}(U_i),
$$
where $I_t$ is the complex structure from Lemma \ref{holomorphic moser}. In fact, since every automorphism $\phi_i$ commutes with the projection $\pi^{-1}(U_i)\times \C\to \C$, it induces a biholomorphism
$$
    \phi_i\colon (\pi^{-1}(U_i)\times\C,I_{tw})\to \pi^{-1}(U_i)\times\C.
$$
Here $I_{tw}$ is the complex structure introduced in Definition \ref{definition of dg.tw}.

\hfill

{\bf Step 3:} It remains to prove that the maps $\phi_i$ can be glued together to a global isomorphism of families $\mathfrak{X}_{deg.tw} \to \mathfrak{X}_{\Sha \mathrm{T}}$. Define the vector fields $w_{ij}$ on $\pi^{-1}(U_{ij})\times\C$ by
$$
\iota_{w_{ij}}\sigma = t\pi^*(a_{ij}).
$$
Let $\phi_{ij}$ be the exponential of $w_{ij}$. Then $\phi_{ij}$ is a holomorphic automorphism of $\pi^{-1}(U_{ij})\times\C$. Since $a_{ij} = a_j-a_i$ one has
$$
\phi_{ij} = \phi_j\circ\phi_i^{-1}.
$$
By Proposition \ref{prop_ShT_family} the manifold $\mathfrak{X}_{\Sha\mathrm{T}}$ is the twist of $X\times\tSha$ by the cocycle $\{\phi_{ij}\}$. Therefore the maps $\phi_i$ glue together to a global biholomorphism $\Phi \colon \mathfrak{X}_{deg.tw}\to \mathfrak{X}_{\Sha\mathrm{T}}$.\qed

%%%%%%%%%%%%%%%%%%%%
%%%%%%%%%%%%%%%%%%%%

\subsection{The period map of a Shafarevich--Tate family}\label{period map}

Here we will describe the period map of a Shafarevich--Tate family. That can be viewed as a generalization of a Markman's result \cite[Thm. 7.11]{markman2014lagrangian}.

The following definition is standard. Let $X$ be a hyperk\"ahler manifold. Its \emph{period domain} $\mathcal{D}$ is defined as 
\[
\mathcal{D}:= \{[x] \ | \ q(x,x)=0 \text{ and } q(x, \overline{x})>0\} \subset \mathbb{P}(H^2(X, \C)) 
\]
where $q$ is the BBF form. Let $p \colon \mathcal{X} \to T$ be a holomorphic family of hyperk\"ahler manifolds over a connected simply connected base $T$. One can associate the \emph{period map} $\mathcal{P}_T \colon T \to \mathcal{D}$ to $p$. This is the holomorphic map defined as 
\[
t \mapsto [H^{2,0}(p^{-1}(t))].
\]

Its importance is manifested by the Torelli theorem: let ${Teich}(X)$ be the Teichm\"uller space of $X$, i.e., the space of complex structures of hyperk\"ahler type on $X$ modulo isotopies. Then  the period map
\[
\mathcal{P} \colon Teich(X) \to \mathcal{D}
\]
is generically injective on connected components of $Teich(X)$ \cite{huybrechts2012global}.

\begin{prop}
Let $\mathfrak{X}_{deg.tw} \to \C$ be the degenerate twistor family of a Lagrangian fibration $\pi \colon X \to B$. Then its period map $\mathcal{P}_{deg.tw} \colon \C \to \mathcal{D}$ is injective. The image of $\mathcal P_{deg.tw}$ is the entire curve that is obtained as the intersection of $\mathcal{D}$ with a projective line lying on the quadric $\{q(x,x) = 0\}$.
\end{prop}
\begin{proof}
Let $\alpha$ be a K\"ahler form on $B$. The period map of the degenerate twistor family is given by
$$
\mathcal P_{deg.tw}(t) = [\sigma + t\pi^*\alpha]
$$
Let $\eta$ denote the class of $\pi^*\alpha$ in $H^2(X,\C)$. Consider the projective line $L:=\mathbb P(\{\sigma,\eta\})$ lying in $\mathbb P(H^2(X,\C))$. The map $\mathcal P_{deg.tw}$ sends $\C$ isomorphically to $L\,\setminus\, [\eta]$. All points of this line satisfy $q(x,x) = 0$. The point corresponding to $\eta$ does not lie in the period domain $\mathcal D$ since $q(\eta,\overline{\eta}) = 0$. Therefore, the image of $\mathcal P_{deg.tw}$ is $L\cap\mathcal D$.
\end{proof}

\begin{rmk}
One may think of $\tSha$ as an entire curve in the Teichm\"uller space of $X$ and $\Sha^0$ as its image inside the moduli space $\mathcal{M}$. The moduli space $\mathcal{M}$ can be defined as $Teich(X)$ modulo the mapping class group of $X$. It is known that the action of the mapping class group on $Teich(X)$ is very non-discrete, therefore, $\mathcal{M}$ is non-Hausdorff. This phenomenon is often seen already at the level of Shafarevich--Tate groups (Theorem \ref{density}) 
\end{rmk}

%%%%%%%%%%%%%%%%%%%%%%%%%%%%%%%%%%%
%%%%%%%%%%%%%%%%%%%%%%%%%%%%%%%%%%%%%%%%%%%
%%%%%%%%%%%%%%%%%%%%%%%%%%%%%%%%%%%%%%%%%%%%%%%

\section{Connected component of unity of a Shafarevich--Tate group}\label{connected component}

\subsection{The sheaf $\Gamma$}

Consider the exponential exact sequence
\[
0 \to \Z_X \to \O_X \to \O^{\times}_X \to 0
\]
Apply the higher direct image functor $R^{\bullet}\pi_*(-)$ to it. Since the fibers of $\pi$ are connected and proper, we have $\pi_*\Z_X \simeq \Z_B$, $\pi_*\O_X \simeq \O_B$, and $\pi_*\O^{\times}_X \simeq \O^{\times}_B$. The sequence
\[
0 \to \pi_*\Z_X \to \pi_*\O_X \to \pi_*\O^{\times}_X \to 0
\]
is therefore exact. Hence we get a long exact sequence of sheaves:
\[
0 \to R^1\pi_*\Z_X \to R^1\pi_*\O_X \to R^1\pi_*\O^{\times}_X \to R^2\pi_*\Z_X \to\ldots
\]

\begin{prop}\label{R^1pi_*Z maps to Gamma}
The isomorphism $$\tilde{\omega}\colon R^1\pi_*\O_X \to \pi_*T_{X/B}$$ from Proposition \ref{Matsushita isomorphism} sends $R^1\pi_*\Z_X \subset R^1\pi_*\O_X$ into $\Gamma:=\ker(\pi_*T_{X/B}\to Aut^0_{X/B})$.
\end{prop}
\begin{proof}
We need to prove that the composition of maps 
$$
R^1\pi_*\mathbb Z\xrightarrow{i} R^1\pi_*\mathcal O_X\xrightarrow{\widetilde\omega} \pi_*T_{X/B}\xrightarrow{\varepsilon} Aut^0_{X/B}
$$
vanishes. Let $b\in B^\circ$ and $F_b$ be a smooth fiber. Then $(R^1\pi_*\O_X)_b = H^1(F_b, \O_{F_b})$, $(\pi_*T_{X/B})|_b = H^0(F_b, TF_b)$, and $\widetilde{\omega}_b$ is the map
\[
H^{0,1}(F_b) \to H^{1,0}(F_b)^{*} \cong H^0(F_b, TF_b),
\]
given by the polarization $[\omega]_{F_b}$ on the abelian variety $F_b$. By Corollary \ref{relative polarisation} we can choose $\omega$ in such a way that this polarization is integral, so $\widetilde{\omega}_b$ maps $H^1(F_b, \Z) \subset H^1(F_b, \O_{F_b})$ to $\Gamma_b = H_1(F_b, \Z) \subset H^0(F_b, TF_b)$.

It follows that the statement of the lemma holds after restriction to $B^{\circ}$. Consider a local section $\tau$ of $R^1\pi_*\Z_X$ over an open subset $U \subset B$. Denote $U \bigcap B^{\circ}$ by $U^{\circ}$. We have seen that restriction of the automorphism $(\varepsilon\circ\widetilde{\omega} \circ i)(\tau)$ to $\pi^{-1}(U_\circ)$ is trivial. An automorphism that is trivial on a dense open subset is trivial. Hence, the vertical automorphism $(\varepsilon \circ \widetilde{\omega} \circ i)(\tau)$ of $\pi^{-1}(U)$ is trivial.
\end{proof}

We have constructed a morphism of sheaves $\alpha\colon R^1\pi_*\Z_X \to \Gamma$. Note that it is injective, since it is a composition of an isomorphism $\widetilde{\omega} \colon R^1\pi_*\O_X \to \pi_*T_{X/B}$ and an embedding $i \colon R^1\pi_*\Z_X \to R^1\pi_*\O_X$.

There is the following diagram with exact rows:
\[
\xymatrix{
0 \ar[r]& \Gamma \ar[r]& \pi_*T_{X/B} \ar[r]^{\varepsilon}& Aut^0_{X/B} \ar[r]& 0\\
0 \ar[r]& R^1\pi_*\Z_X \ar[r]^{i} \ar[u]^{\alpha} & R^1\pi_*\O_X \ar[u]_{\widetilde{\omega}} \ar[r]& R^1\pi_*\O^{\times}_X 
}
\]

\hfill

Recall the \textit{local invariant cycle theorem}:

\begin{prop}[{\cite[Cor. 6.2.9]{beilinson2018faisceaux}}]
\label{local invariant cycle}
Let $f \colon X \to Y$ be a proper map of complex algebraic varieties with $X$ non-singular. For a subset $U \subset Y$, denote by $U^{\circ}$ the set of non-critical values of $f$ in $U$. Then for any $b \in Y$ there exists a small ball $U_b \subset Y$ with the center $b$ such that the following holds.
\begin{itemize}
\item[(1)]For any $i$ the following groups are isomorphic: $H^i(f^{-1}(b), \Q) \simeq H^i(f^{-1}(U_b), \Q)$;
\item[(2)] For any point $b^{\circ} \in U_b^{\circ}$ the cohomology group $H^i(f^{-1}(U_b),\Q)$ surjects onto the local monodromy invariants $H^i(f^{-1}(b^\circ),\Q)^{\pi_1(U_b^\circ)} = R^i\pi_*\Q_X(U_b^\circ)$.
\end{itemize}
\end{prop}

\begin{rmk}
Let $\pi \colon X \to B$ be a Lagrangian fibration on a hyperk\"ahler manifold. For each $b \in B$ there exists an open subset $U \subset B$ containing $b$ such that $\pi^{-1}(U)$ is a smooth algebraic variety \cite[Cor. 3.4]{soldatenkov2021moser}. Therefore, the  local invariant cycle theorem can be applied in this situation.
\end{rmk}

Let $\Gamma_{\Q}$ denote the sheaf $\Gamma \o \Q_B$.

\begin{prop}\label{Gamma is almost R^1pi_*Z}
The map $\alpha\otimes\Q\colon R^1\pi_*\Q\to \Gamma_\Q$ is an isomorphism.
\end{prop}
\begin{proof}
\textbf{Step 1.} First, notice that the lemma holds after restriction to the non-critical set $B^{\circ} \subset B$. Indeed, if $F_b = \pi^{-1}(b)$ is a smooth fiber, then 
\[
\alpha_b \colon (R^1\pi_*\Q_X)_b = H^1(F_b, \Q) \to (\Gamma_{\Q})_b = H_1(F_b, \Q)
\]
is the map given by the polarization on the abelian variety $F$. Hence it is an isomorphism. Note that in general the same might not hold with integral coefficients.

\hfill

\textbf{Step 2.} Take a point $b\in B$ and let $U_b \subset B$ be as in Proposition \ref{local invariant cycle}. We need to prove that the map 
$$
    \alpha_{U_b}\colon H^1(U_b,\Q) \to \Gamma_\Q(U_b)
$$
is an isomorphism. We already know that it is injective, hence we only need to check surjectivity. Take a local section $\gamma \in \Gamma_\Q(U_b)$. Let $\gamma^\circ$ be its restriction to $U_b^\circ$. By {\bf Step 1} there exists $\beta^\circ \in H^0(U_b^\circ, R^1\pi_*\Q_X)$, such that $\alpha_{U_b^\circ}(\beta^\circ) = \gamma^\circ$. The section $\beta^\circ$ can be lifted to a section $\beta$ of $H^0(U_b,R^1\pi_*\Q_X) = H^1(\pi^{-1}(U_b),\Q)$ by the local invariant cycle theorem. The local sections $\alpha_{U_b}(\beta) \in \Gamma_\Q(U_b)$ and $\gamma\in \Gamma_\Q(U_b)$ coincide after restriction to $U_b^{\circ}$ by their construction. The sheaf $\Gamma_\Q$ is a subsheaf of the locally free sheaf $R^1\pi_*\mathcal O_X$, hence $\alpha_{U_b}(\beta)$ and $\gamma$ coincide.
\end{proof}

%%%%%%%%%%%%%%
%%%%%%%%%%%%%%

\subsection{The connected component of unity of Shafarevich--Tate groups}\label{Sha^0}
The goal of this Subsection is to describe the group $\Sha^0$, i.e., the connected component of unity of the group $\Sha$. This group fits into the short exact sequence
$$
    H^1(B,\Gamma) \to \tSha \to \Sha^0\to 0.
$$

In the previous Subsection, we constructed an injective map $\alpha\colon R^1\pi_*\mathbb Z\to \Gamma$ and proved that $\alpha$ is an isomorphism after tensoring with $\Q$ (Proposition $\ref{Gamma is almost R^1pi_*Z}$). Although we are mostly concerned with the sheaf $\Gamma$, the sheaf $R^1\pi_*\mathbb Z$ is often much easier to understand.

Define the subspace $W\subset H^2(X,\C)$ as the kernel of the restriction map 
$$
H^2(X, \C) \to H^0(B, R^2\pi_*\C_X).
$$
This is a Hodge substructure in $H^2(X, \Z)$ that can be described as
\begin{equation}
\label{definition of W}
W_\mathbb Z = \bigcap_{b \in B} \ker (H^2(X, \mathbb Z) \to H^2(F_b, \mathbb Z)),
\end{equation}
Let us denote $W_{\Q}:= W \bigcap H^2(X, \Q)$.

As before, let $\eta\in H^2(X,\mathbb Z)$ be the pullback of the class of a hyperplane section. 
Let $\eta^{\perp}$ be the orthogonal to $\eta$ with respect to the BBF form.

\begin{prop}\label{W is big enough}
Let $W$ be as above. Then we have the following chain of inclusions
$$
\{\eta,[\sigma],[\bar\sigma]\}\subset W \subset \eta^\perp.
$$
\end{prop}

\begin{proof}

The class $\eta$ is the Chern class of the line bundle $\pi^*\mathcal O_{\mathbb P^n}(1)$. This line bundle restricts trivially to every fiber of $\pi$, hence $\eta$ is contained in $W$. By Proposition \ref{restriction} the kernel of the restriction map 
$H^2(X,\mathbb C)\to H^2(F_b,\mathbb C)$ to a smooth fiber $F_b$ is $\eta^\perp$. We obtain that $W\subset\eta^\perp$.

We are left to prove that $[\sigma]$ is contained in $W$. Indeed, $W$ is a Hodge substructure, thus it will imply that $[\overline{\sigma}]$ is in $W$ as well. 
 The fibration $\pi$ is Lagrangian, hence the restriction of the $2$-form $\sigma$ to every smooth fiber of $\pi$ vanishes. Therefore the class $[\sigma]$ in $H^2(X, \C)$ restricts trivially to a smooth fiber.

Let $F_b$ be a singular fiber. By \cite[Thm. 6.9]{clemens1977degeneration} there is a neighborhood $U_b\subset B$ of $b$ and a smooth retraction of $\pi^{-1}(U_b)$ to $F_b$ i.e. a one-parameter family of diffeomorphisms $f_t$ such that
$$
f_t\colon \pi^{-1}(U_b)\to \pi^{-1}(U_b)\quad\quad\quad f_0 = \mathrm{id}_{\pi^{-1}(U_b)},\; \mathrm{im}(f_1) = F_b,\; f_t|_{F_b} = \mathrm{id}_{F_b}\, \forall t
$$

One can see that $f_0^*\sigma = \sigma$ and $f_1^*\sigma = 0$. Let $\xi_t$ denote the tangent vector field to the one-parameter family $f_t$. Compute
$$
\sigma|_{\pi^{-1}(U_b)} = -(f_1^*\sigma - f_0^*\sigma) = -\int\limits_0^1\frac{d}{dt}f_t^*\sigma = -\int\limits_0^1f_t^*\mathsf{L}_\xi \sigma = - d\left(\int_0^1 f_t^*\iota_\xi\sigma\right).
$$
We conclude that the form $\sigma$ becomes exact after the restriction to $\pi^{-1}(U_b)$. Therefore, the class of $\sigma$ is indeed contained in $W$.
\end{proof}
\begin{rmk}
If all fibers of $\pi$ are irreducible, then $\ker \left( H^2(X, \Q) \to H^2(F_b, \Q) \right)$ does not depend on $b$. In particular, $W = \eta^{\perp}$. We learned the proof of this fact from Claire Voisin. Let $\nu\colon \tilde F_b\to F_b$ be a resolution of singularities of $F_b$. Denote the inclusion of $F_b$ into $X$ by $\iota$. By \cite[Lemma 1.5, Remark 1.6]{voisin1992stabilite}, the kernel of the pullback map $(\iota\circ\nu)^*\colon H^2(X,\Q)\to H^2(\tilde F_b,\Q)$ is equal to $\{\alpha\in H^2(X,\Q)\:|\: [F_b]\cdot \alpha = 0\}$. This subspace does not depend on the choice of $b$, hence it is equal to $\eta^\perp$. 

We claim that $\ker \iota^* = \ker(\iota\circ\nu)^*$. This follows from Deligne's theory of mixed Hodge structures. The map $\iota^*$ is a morphism of mixed Hodge structures, hence strict \cite[Corollary 3.6]{peters2008mixed}. In particular, $\operatorname{im} \iota^*\cap W_1H^2(F) = \{0\}$. By \cite[Corollary 5.42]{peters2008mixed} we have
$$
\ker \nu^*\colon H^2(F)\to H^2(\tilde F) = W_1H^2(F).
$$
The desired claim follows.
\end{rmk}

\begin{prop}\label{spectrals}
Consider the map 
$$
H^1(B,R^1\pi_*\mathbb Z)\to H^1(R^1\pi_*\mathcal O_X)\simeq\tSha
$$
induced by the embedding $R^1\pi_*\mathbb Z\to R^1\pi_*\mathcal O_X$. There are canonical isomorphisms 
$$\tSha \simeq H^{0,2}(X)$$  
and 
$$H^1(B,R^1\pi_*\mathbb Z) \simeq W_{\mathbb Z}/\eta$$
that fit into the commutative diagram
$$
\xymatrix{
H^1(B,R^1\pi_*\mathbb Z) \ar[r]\ar[d]_{\simeq}& \tSha \ar[d]^{\simeq}\\
W_{\mathbb Z}/\eta \ar[r]^{p} & H^{0,2}(X)
}
$$
Here $p$ is induced by the Hodge projection $W_\mathbb Z\subset H^2(X,\mathbb Z)\to H^{0,2}(X)$.
\end{prop}
\begin{proof}
Consider the Leray spectral sequence for the sheaf $\mathcal O_X$. The terms of the second page of this spectral sequence can be computed using Proposition \ref{Matsushita isomorphism} as
\[
E^{p,q}_2=H^p(B, R^q\pi_*\O_X) = H^{p}(B, \Omega^q_B) = \begin{cases} \C \text{ if } p=q,\\ 0 \text{ otherwise. } \end{cases}
\]
This spectral sequence degenerates on the second page. It converges to the cohomology groups of the sheaf $\mathcal O_X$. Therefore,
$$
H^2(X,\mathcal O_X) \cong E_2^{1,1} = H^1(R^1\pi_*\mathcal O_X) \cong \tSha.
$$
%Consider the second page of the Leray spectral sequence for the sheaf $R^1\pi_*\mathbb Z$.

%\begin{sseqdata}[name = sseqzed, xscale = 2.4, classes = {draw = none}, cohomological Serre grading]
%\class["0"](-1,2)
%\class["0"](-1,1)
%\class["0"](-1,0)
%\class["H^0(R^2\pi_*\mathbb Z)"](0,2)
%\class["H^0(R^1\pi_*\mathbb Z)"](0,1)
%\class["\mathbb Z"](0,0)
%\class["H^1(R^1\pi_*\mathbb Z)"](1,1)
%\class["0"](1,0)
%\class["H^2(R^1\pi_*\mathbb Z)"](2,1)
%\class["\mathbb Z"](2,0)
%\class["0"](3,0)
%\end{sseqdata}
%\begin{sseqpage}[name = sseqzed, page = 2]
%\classoptions[blue](1,1)
%\classoptions[blue](3,0)
%\classoptions[blue](-1,2)
%\d[blue]2(-1,2)
%\d[blue]2(1,1)
%\end{sseqpage}

%The natural map $E^{p,0}_2 = H^p(B, \Z) \to H^p(X, \Z)$ is given by the pullback map $\pi^*$. The pullback map on integral cohomology is injective for every surjective map of compact K\"ahler manifolds (see e.g. \cite{Vois}). This implies that $E^{p,0}_2 = E^{p,0}_{\infty}$. Hence, the differential $d_p \colon E^{1,1}_p \to E^{p+1,0}_p$ vanishes for all $p>1$. Therefore

Let us now consider the second page of the Leray spectral sequence for the sheaf $\mathbb Z_X$. The differential $d_2\colon E_2^{1,1}\to E_2^{3,0}$ vanishes because $E_2^{3,0} = H^3(\mathbb P^n,\mathbb Z) = 0$. All the other differentials starting in $E^{1,1}$ vanish because their targets are in negative grading. Therefore,
$$
E^{1,1}_\infty = E^{1,1}_2 = H^1(B,R^1\pi_*\mathbb Z).
$$
The filtration on the Leray spectral sequence induces a filtration $F^\bullet H^2(X,\mathbb Z)$ on the cohomology groups of $X$. We obtain the following short exact sequences:
\begin{gather*}
0 \to \Z \xrightarrow{\cdot\eta} F^1H^2(X, \Z) \to H^1(B, R^1\pi_*\Z) \to 0\\
0 \to F^1H^2(X, \Z) \to H^2(X, \Z) \to H^0(B, R^2\pi_*\Z).
\end{gather*}

The second exact sequence implies that $F^1H^2(X,\Z) = W_{\Z}$. We see from the first exact sequence that $H^1(B, R^1\pi_*\Z_X) \simeq W_{\Z}/\Z \eta$. The commutativity of the diagram in the statement of the proposition follows from functoriality properties of Leray spectral sequences. 
\end{proof}

\begin{cor}
\label{description of sha zero}
Let $\pi\colon X\to B$ be a Lagrangian fibration. Define the group $\widehat{\Sha}^0$ to be the cokernel of the map
$$
W_\mathbb Z/\eta \xrightarrow{p} H^{0,2}(X).
$$
Then the group $\Sha^0$ is a quotient of $\widehat{\Sha}^0$ by a finite group.
\end{cor}
\begin{proof}
The map $\alpha\colon R^1\pi_*\mathbb Z\to \Gamma$ induces the map $H^1(\alpha)\colon H^1(R^1\pi_*\mathbb Z)\to H^1(B,\Gamma)$. The map $H^1(\alpha)$ has a finite cokernel by Proposition \ref{Gamma is almost R^1pi_*Z}. By Proposition \ref{spectrals} there is the following commutative diagram
$$
\xymatrix{
W_\mathbb Z/\eta \ar[r]\ar[d]_{H^1(\alpha)}& H^{0,2}(X) \ar[r]\ar[d]^{\simeq}& \widehat{\Sha}^0\ar[r]\ar[d]& 0\\
H^1(B,\Gamma) \ar[r]& \tSha \ar[r]& \Sha^0\ar[r]&0
}
$$
The kernel of the map $\widehat{\Sha}^0\to \Sha^0$ is isomorphic to $\operatorname{coker}(H^1(\alpha))$ by the snake lemma. Hence it is finite.
\end{proof}

%%%%%%%%%%%%%%%%%%%%%%
%%%%%%%%%%%%%%%%%%%%%%

\section{K\"ahlerness and projectivity properties}\label{applications}

In this Section we study K\"ahlerness and algebraicity properties of degenerate twistor deformations. Unfortunately we are not able to prove that a denegerate twistor deformation of a hyperk\"ahler manifold $X$ is always K\"ahler. We need to impose some additional assumption on $X$. We call manifolds that do not satisfy this assumption \textit{M-special}.

\subsection{M-special Hodge structures}\label{M-special}

In this Subsection, we summarize several results about Hodge structures of K3 type. In particular, we introduce and discuss M-special Hodge structures. A pure Hodge structure  $W$ is said to be \textit{of K3 type} if it is of weight $2$ and $\dim W^{2,0}=1$.

Let $W$ be a $\Z$-Hodge structure of weight $2$ and $\mathbb{K} \subset \C$ a subring. We will write $W^{1,1}_{\mathbb{K}}: = W_{\mathbb{K}} \bigcap W^{1,1}$ and $W^{2,0+0,2}_{\mathbb{K}}:=W_{\mathbb{K}} \bigcap \left(W^{2,0} \oplus W^{0,2} \right)$. We denote the Hodge projection by $p\colon W\to W^{0,2}$. Let $q$ be a polarisation on $W$. If $W$ is a Hodge structure of K3 type, then the projection $p\colon W\to W^{0,2}$ is given by the map
$$
v\to q(v,\sigma)\bar{\sigma}
$$
for an appropriate choice of $\sigma \in W^{2,0}$.

We define the transcendental Hodge substructure $T\subseteq W$ to be the orthogonal complement to $W^{1,1}_\mathbb Z$. The lattice $T$ has the rank at least two because
$$
    W^{2,0 + 0,2}_{\mathbb R} \subseteq T_{\mathbb R}.
$$
Note that the restriction of the Hodge projection $p$ to $T$ is an isomorphism to its image.

\begin{df}\label{definition of M-special}
Let $W$ be a $\Z$-Hodge structure of K3 type. It is said to be \textit{M-special} if $W_{\Z}^{2,0+0,2} \neq \{0\}$.
\end{df}

\begin{rmk}
\label{subHS is M-special}
Let $V\subset W$ be an embedding of Hodge structures of K3 type. Then $V$ is M-special if and only if so is $W$.
\end{rmk}

The following definition is a straightforward generalization of \cite[Def. 1]{markman2014lagrangian}.

\begin{lemma}
\label{first}
Let $W$ be a polarized $\mathbb Z$-Hodge structure of K3 type. The following are equivalent.
\begin{enumerate}
    \item The subspace $W^{1,1}$ is rational ("Picard rank is maximal").
    \item The subspace $W^{2,0+0,2}$ is rational.
    \item The rank of the transcendental Hodge substructure $T$ is two.
    \item The image of $W_\mathbb Z$ under the Hodge projection $p\colon W_\mathbb Z\to W^{0,2}$ is a discrete subgroup of $W^{0,2}$.
\end{enumerate}

\end{lemma}
The proof of this lemma is straightforward and is left to the reader. If one of the conditions of the lemma holds, then $W$ is M-special. However, the converse is not true.

\begin{lemma}
[{\cite[Lemma 5.5]{markman2014lagrangian}}]
\label{second}
Let $W$ be a polarized $\mathbb Z$-Hodge structure of K3 type that is not M-special. Assume that $\rk T\ge 3$. Then for every sublattice $T'\subset T$ of corank one, the group $p(T')$ generates $W^{0,2}$ as an $\mathbb R$-vector space.
\end{lemma}
\begin{proof}
Let $T'\subset T$ be a sublattice of corank one. Suppose that $p(T')$ does not generate $W^{0,2}$ over $\mathbb R$. In that case there exist real numbers $a, b$ such that
$$
a \cdot q(\operatorname{Re}(\sigma),t) + b\cdot q(\operatorname{Im}(\sigma),t) = 0\quad \forall t\in T'.
$$
Hence the vector $l = a\cdot\operatorname{Re}(\sigma) + b\cdot\operatorname{Im}(\sigma) \in W^{2,0+0,2}_\mathbb R$ is orthogonal to $T'$. Observe that $l^\perp = (W^{1,1}_\mathbb Z + T')\otimes\mathbb R$ since $(W^{1,1}_\mathbb Z+ T')\o\mathbb R\subseteq l^\perp$ and the dimensions of the two spaces coincide. The subspace $l^\perp$ is therefore rational. Hence, a multiple of $l$ is rational. This contradicts the assumption that $W$ is not M-special.
\end{proof}

\begin{lemma}
[{\cite[Lemma 5.5]{markman2014lagrangian}}]
\label{third}
Let $T\subset \mathbb R^2$ be a finitely generated subgroup of $\mathbb R^2$ of rank at least three. Assume that every subgroup $T'\subset T$ of corank one generates $\mathbb R^2$ over $\mathbb R$. Then $T$ is dense in $\mathbb R^2$.
\end{lemma}

\begin{thrm}
[{\cite[Lemma 5.4, 5.5]{markman2014lagrangian}}]
\label{M-special theorem}
Let $W$ be a polarized $\mathbb Z$-Hodge structure of K3 type. Then the following are equivalent
\begin{enumerate}
    \item The Hodge structure $W$ is M-special.
    \item The image of $W_\mathbb Z$ under the Hodge projection $p\colon W\to W^{0,2}$ is not dense in $W^{0,2}$.
\end{enumerate}
\end{thrm}

\begin{proof}
{$\mathbf{(1)\Rightarrow(2)}$} Suppose that $W$ is M-special. Choose a non-zero element $l\in W^{2,0+0,2}_\mathbb Z$. There exists a unique element $\sigma\in W^{2,0}$ such that $\operatorname{Re}\sigma = l$. The projection $p\colon W\to W^{0,2}$ is identified with the map $v\in W\mapsto q(\sigma,v)$ up to scaling. For every $v \in W_{\mathbb Z}$  the number $\operatorname{Re}(q(\sigma, v)) = q(l,v)$ is an integer. Hence, there exists a constant $\lambda \in \R$ such that $\operatorname{Re}(z) \in \Z\lambda$ for every vector $z$ in $p(W_{\Z})$.
Therefore, $p(W_\mathbb Z)$ cannot be dense in $W^{0,2}$.

\hfill

{$\mathbf{(2)\Rightarrow(1)}$} Suppose $W$ is not M-special. In that case the rank of $T$ is at least three (Lemma \ref{first}). Lemma \ref{second} implies that for every sublattice $T'\subset T$ of corank one, the group $p(T')$ generates $W^{0,2}$ over $\mathbb R$. It follows from Lemma \ref{third} that $p(T_\mathbb Z) = p(W_\mathbb Z)$ is dense in $W^{0,2}$.
\end{proof}

%%%%%%%%%%%%%%%%%%%%%%%%%%%%%%%
%%%%%%%%%%%%%%%%%%%%%%%%%%%%%%%

\subsection{K\"ahlerness of degenerate twistor deformations}\label{kahler}

Let $\pi\colon X\to B$ be a Lagrangian fibration. Recall that  we defined the Hodge substructure $W \subset H^2(X, \Z)$ as 
$$
W:= \ker \left( H^2(X, \Z) \to H^0(B, R^2\pi_*\Z_X) \right).
$$
 The Hodge structure $W$ is of K3 type by Proposition \ref{W is big enough}. Let us denote the Hodge projection by $p\colon H^2(X)\to H^{0,2}(X)$.
\begin{thrm}
\label{density}
Let $\pi\colon X\to B$ be a Lagrangian fibration, $\Sha^0$ the connected component of unity of its Shafarevich--Tate group $\Sha$. Then $\Sha^0$ is Hausdorff if and only if $X$ is of maximal Picard rank. In this case $\Sha^0$ is isomorphic to an elliptic curve.
\end{thrm}
\begin{proof}

The group $\Sha^0$ is Hausdorff if and only if the topological group $\widehat{\Sha}^0$ defined as $H^{0,2}(X)/p(W_\mathbb Z)$ is Hausdorff. Indeed, the latter group is a finite unramified cover of $\Sha^0$ (Corollary \ref{description of sha zero}). We see that $\Sha^0$ is Hausdorff if and only if $p(W_\mathbb Z)\subset H^{0,2}(X)$ is a discrete subgroup. Apply Lemma \ref{first} to the Hodge structure $W/\eta$. We obtain that $p(W_\mathbb Z)$ is discrete in $H^{0,2}(X)$ if and only if $W^{2,0+0,2} = H^{2,0}(X)\oplus H^{0,2}(X)$ is a rational subspace. This is equivalent to saying that $H^{1,1}(X)$ is rational, i.e., $X$ has the maximal Picard rank (Lemma \ref{first}). If $p(W_\mathbb Z)$ is discrete, then $p(W_\mathbb Z)$ is a lattice in $H^{0,2}(X)$ of rank two and $\Sha^0$ is an elliptic curve.
\end{proof}

\begin{df}
\label{M-special manifolds}
Let $X$ be a hyperk\"ahler manifold. It is called {\em M-special} if $$(H^{2,0}(X)\oplus H^{0,2}(X))\cap H^2(X,\mathbb Z)\ne\{0\}.$$
\end{df}

In other words, a hyperk\"ahler manifold is {\em M-special} if the Hodge structure on its second cohomology is M-special (Definition \ref{definition of M-special}).

\begin{lemma}
\label{speciality of X and W}
Let $\pi\colon X\to B$ be a Lagrangian fibration on a projective hyperk\"ahler manifold $X$. Then $X$ is M-special if and only if the Hodge structure $W_\mathbb Z/\eta$ is M-special.
\end{lemma}
\begin{proof}
Let $l\in H^{1,1}(X)_\mathbb Z$ be an ample class. Since $q(l,\eta)>0$ the pairing with $l$ induces a non-zero functional on $W$. Thus, $l^\perp\cap W$ is a hyperplane in $W$ not containing $\eta$. We obtain the direct sum decomposition $W = \eta\oplus (l^\perp\cap W)$. The subspace $l^\perp\cap W$ is a polarized Hodge substructure of $H^2(X,\mathbb Z)$ isomorphic to $W/\eta$. It follows from Remark \ref{subHS is M-special} that $l^\perp\cap W$ is M-special if and only if so is $H^2(X,\mathbb Z)$.
\end{proof}

\begin{rmk}
The statement of Lemma \ref{speciality of X and W} in general does not hold for non-projective manifolds. However, it is still true that if $X$ is M-special then so is $W/\eta$.
\end{rmk}

We are now ready to prove the main theorem of this Section.

\begin{thrm}\label{Kahlerness}
Let $\pi \colon X \to B$  be a Lagrangian fibration and $\mathfrak{X}_{\Sha\mathrm{T}} \to \tSha$ the Shafarevich--Tate family (Subsection \ref{ShT-family}). Assume that $X$ is projective and not M-special.
Then the members of the family $X^s \subset \mathfrak{X}_{\Sha\mathrm T}$ are K\"ahler for each $s \in \tSha$.
\end{thrm}

\begin{proof} The zero fiber $X^0$ is K\"ahler by assumption. K\"ahlerness is an open condition. Thus there exists an open subset $U \subset \tSha$ such that for every $s \in U$ the twist $X^s$ is K\"ahler \cite[Thm. 9.23]{voisin2007hodge}. The manifolds $X^{s}$ and $X^{s+\lambda}$ are isomorphic for each $\lambda$ in the subgroup $\Lambda:=\mathrm{im}(H^1(B,\Gamma)) \subset \tSha$. The vector spaces $p(W_\mathbb Z)\o\Q$ and $\Lambda\o\Q$ coincide (Proposition \ref{Gamma is almost R^1pi_*Z}). By Lemma \ref{speciality of X and W} the Hodge structure $W/\eta$ is not M-special. It follows from Theorem \ref{M-special theorem} that $p(W_\mathbb Z)$ is dense in $\tSha \simeq H^{0,2}(X)$. Hence $\Lambda$ is dense in $\tSha$. We obtain that for every $s\in \tSha$ there exists $s' \in U$ such that $X^s \simeq X^{s'}$. Consequently, $X^{s}$ is K\"ahler for every $s$.
\end{proof}

\begin{rmk}
The statement of Theorem \ref{Kahlerness} remains true for non-projective manifolds if we assume that the Hodge structure $H^2(X,\mathbb Z)/\eta$ is not M-special.
\end{rmk}

\begin{rmk}
As we mentioned in the introduction, the first-named author proved a stronger version of Theorem \ref{Kahlerness} after a preliminary version of this work appeared on ArXiv \cite{abasheva2024shafarevich}. Theorem A of \cite{abasheva2024shafarevich} says that the conclusion of Theorem \ref{Kahlerness} holds even if $X$ is $M$-special. Moreover, instead of taking a twist by $s \in \Sha^0$ one can take any $s \in \Sha$ such that $N \cdot s \in \Sha^0$ for some $N \in \mathbb{N}$. 

This leads to the following open question. Consider an element $s \in \Sha$ such that its image in the group of connected components $\Sha/\Sha^0$ is of infinite order. When is the twist $X^s$ a K\"ahler manifold? Note that from results of Subsection \ref{symplectic} we know that $X^s$ possesses a holomorphic symplectic form. It would be interesting to know if one can obtain examples of non-K\"ahler holomorphic symplectic manifolds, such as in \cite{guan1994examples,bogomolov1996guan}.

We do not know, whether there exists a Lagrangian fibration $X \to B$ such that the group of connected components of its Shafarevich-Tate group $\Sha/\Sha^0$ is not torsion. In Section \ref{obstruction} below we show that $\Sha/\Sha^0$ is finite if $b_3(X)=0$, see the proof of Theorem \ref{when a vanishes}.
\end{rmk}

\begin{rmk}
Consider the set $\mathcal D_\eta$ of Hodge structures on $H^2(X,\mathbb Z)$ such that $\eta$ is in $H^{1,1}(X)$. It can be identified with a complex manifold of complex dimension $b_2-3$. The set of Hodge structures in $\mathcal D_\eta$ such that $H^2(X,\mathbb Z)/\eta$ is M-special is a countable union of real analytic subvarieties of real dimension $b_2-3$. 
\end{rmk}

%%%%%%%%%%%%%%%%%%%%%%%%%
%%%%%%%%%%%%%%%%%%%%%%%%%

\subsection{Algebraic points in Shafarevich--Tate families}\label{algebraic points}

In this subsection we will describe the set of projective deformations in the Shafarevich--Tate family.

\begin{lemma}\label{invariants in H^2}
Consider the degenerate twistor family $\mathfrak{X}_{deg.tw} \to \C$ and let $X=X^0$ be the fiber over $0 \in \C$. Let $\alpha$ be a class in $H^{2}(X, \R)$. Then exactly one of the following holds:
\begin{itemize}
\item[(1)] $\alpha \in \eta^{\perp}$;
\item[(2)] there exists a unique $s \in \C$ such that $\alpha \in H^{1,1}(X^s)$.
\end{itemize}
Moreover, $\bigcap_{s \in \C} H^{1,1}(X^s) = (\eta^{\perp})^{1,1}$.
\end{lemma}
\begin{proof}
For each $s\in\C$ the complex structure on $X^s$ is defined by a c-symplectic form with cohomology class $[\sigma_s]=[\sigma_0]+s\eta$. A real class $\alpha$ lies in $H^{1,1}(X^s)$ if and only if it is orthogonal to $\sigma_s$ i.e.
$$
q(\alpha, \sigma_s) = q(\alpha, \sigma_0) + sq(\alpha, \eta) = 0.
$$
Suppose that $\alpha \in H^2(X,\mathbb R) \,\setminus\, \eta^{\perp}$. Then $s_{\alpha} := - q(\alpha, \sigma_0)/q(\alpha, \eta)$
is the unique number satisfying $q(\alpha,\sigma_{s_\alpha}) = 0$. If $\alpha$ is contained in $\eta^\perp$, then it is of type $(1,1)$ for every degenerate twistor deformation.
\end{proof}

\begin{cor}\label{Fujiki-Verbitsky deg.tw}
The set $\mathcal{R}:= \{ s \in \tSha \ | \ X^s \text{ is algebraic }\}$ is at most countable. In particular, a very general member of the Shafarevich--Tate family is non-algebraic. 
\end{cor}
\begin{proof}
If $X^s$ is algebraic, it carries an ample divisor $L$ and $c_1(L)\in H^{1,1}_\mathbb Z(X)$. Moreover, we have the inequality $q(c_1(L), \eta) > 0$ in this case. For every $\alpha \in H^2(X,\Z) \,\setminus\, \eta^{\perp}$ there is only one point $s_{\alpha}$ such that $\alpha \in H^{1,1}(X^{s_{\alpha}})$ (Lemma \ref{invariants in H^2}). Thus, $\mathcal R$ is contained in the set
\[
\{ s_{\alpha} \ | \ \alpha \in H^2(X, \Z) \,\setminus\, \eta^{\perp} \}.
\]
It follows that $\mathcal R$ is at most countable.
\end{proof}

See Theorem \ref{projectivity=torsion} below for a refinement of Corollary \ref{Fujiki-Verbitsky deg.tw}. A similar theorem holds for twistor families $\mathcal{X}_{tw} \to \mathbb P^1$  \cite{verbitsky1996algebraic,fujiki1983primitively}.

\begin{lemma}\label{characterization of projectivity}
Let $\pi \colon X \to B$ be a Lagrangian fibration on a hyperk\"ahler manifold over $B=\mathbb P^n$. The following are equivalent:
\begin{itemize}
\item[(1)] $X$ is projective;
\item[(2)] $\pi$ admits a \textit{rational multisection}, i.e., there exists a subvariety $Z \subset X$ such that $\pi|_Z \colon Z \to B$ is a generically finite morphism;
\item[(3)] there exists a class $\alpha \in H^{1,1}_{\Q}(X)$ such that $q(\alpha, \eta) \neq 0$;
\item[(4)] there exists a class $\omega \in H^{1,1}_{\Q}(X)$ such that $q(\omega, \omega) >0$.
\end{itemize}
\end{lemma}
\begin{proof}
{$\mathbf{(1)\Rightarrow(2)}$} Choose a holomorphic embedding $i \colon X \hookrightarrow \mathbb P^N$. Pick a point $x \in X$ outside of the singular locus of $\pi$. Consider a general linear subspace $L\subset \mathbb P^N$ of codimension $n$ passing through $x$ and transversal to $F_x$. Let $Z$ be a component of $Z' := L \bigcap X \subset X$ which is transversal to the general fiber of $\pi$. Then $Z$ is a rational multisection of $\pi$.

\hfill

{$\mathbf{(2)\Rightarrow(3)}$} Let $Z \subset X$ be a  multisection.  Let $C_0 \subset B$ be a smooth curve and let $C := \pi^{-1}(C_0) \bigcap Z$ be its preimage in $Z$. Consider the homology class $[C] \in H_2(X, \mathbb Z)$ of the curve $C$. The BBF form defines an isomorphism
\[
H^2(X, \Q) \to H^2(X, \Q)^* \simeq H_2(X, \Q).
\]
Let $\alpha \in H^2(X, \Q)$ be the class corresponding to $[C]$ under this isomorphism. Then $q(\alpha, \eta) = \int_{C} \eta > 0$.

\hfill

{$\mathbf{(3)\Rightarrow(4)}$} Let $\alpha$  be a rational $(1,1)$-class such that $q(\alpha, \eta)  \neq 0$. We may assume that $q(\alpha, \eta) > 0$. Let us  find a number $t$ such that $\omega:= \alpha +t\eta$ satisfies $q(\omega, \omega) > 0$. We compute that
\[
q(\alpha+t\eta, \alpha+t\eta) = q(\alpha, \alpha) +2tq(\alpha, \eta),
\]
Therefore, the class $\omega$ is rational and has a positive square for any rational $t > \frac{-q(\alpha, \alpha)}{2q(\alpha, \eta)}$.

\hfill

{$\mathbf{(4)\Rightarrow(1)}$} See \cite[Thm. 3.11]{huybrechts2003compact}.
\end{proof}

\begin{rmk}
If a K\"ahler manifold  is algebraic\footnote{ i.e., the analytification of a smooth proper algebraic variety over $\C$.}, it is necessarily projective by the Moishezon theorem \cite[Thm. 11]{moishezon1966ndimensional}. Therefore, we can replace the first condition in the theorem above with the one that $X$ is algebraic.
\end{rmk}

Denote by $p \colon H^2(X, \C) \to H^{0,2}(X)$ the Hodge projection. Take $\sigma$ to be the holomorphic symplectic form such that $p(\sigma)$ is identified with the class of the Fubini-Study form under the isomorphism $H^{0,2}(X) \simeq \tSha \simeq H^{1,1}(B)$. In the following lemma we will give a characterization of torsion in $\Sha^0$ in terms of the BBF form.

\begin{lemma}
\label{torsion in sha}
Let $\pi\colon X\to B$ be a Lagrangian fibration. Consider a class $t\bar\sigma \in H^{0,2}(X)$, $t\in \C$. Let $[t\bar\sigma]$ denote its image in $\Sha^0$. Then $[t\bar\sigma]$ is torsion if and only if there exists a rational class $l\in W_{\mathbb Q}\subset H^2(X,\mathbb Q)$ such that $q(l,\sigma) = t$ (see formula (\ref{definition of W}) for the definition of $W_{\mathbb Q}$).
\end{lemma}

\begin{proof}
We know from the short exact sequence (\ref{exact seq of shas}) that $\Sha^0 = \Sha/\Lambda$ where 
$\Lambda$ denotes the group $\mathrm{Im}(H^1(B,\Gamma) \to \tSha)$. Hence, the class $[t\bar\sigma]$ is torsion if and only if $t\bar\sigma$ lies in the image of the group $H^1(B,\Gamma)\otimes\mathbb Q$. This is equivalent to saying that $t\bar\sigma$ lies in the image of $W_{\mathbb Q}\subset H^2(X,\mathbb Q)$ under the Hodge projection $H^2(X,\mathbb Q)\to H^{0,2}(X)$ (see Proposition \ref{Gamma is almost R^1pi_*Z}, Proposition \ref{spectrals}). Since $H^{1,1}(X)$ is orthogonal to $H^{2,0}(X)\oplus H^{0,2}(X)$, the projection $H^2(X,\mathbb C)\to H^{0,2}(X)$ is given by the formula
$$
v\to q(v,\sigma)\bar\sigma.
$$
Hence, $t\bar\sigma$ lies in the image of $W_{\mathbb Q}$ under the Hodge projection if and only if there exists a class $l\in W_{\mathbb Q}$ such that $t = q(l,\sigma)$.
\end{proof}

\begin{thrm}
\label{projectivity=torsion}
Let $\pi\colon X\to B$ be a Lagrangian fibration. Suppose that $X$ is projective. Let $s$ be an element of $\Sha^0$. The following are equivalent.
\begin{enumerate}
    \item The Shafarevich--Tate twist $X^s$ of $X$ is projective.
    \item The manifold $X^s$ is K\"ahler and $s$ is a torsion element of $\Sha^0$.
\end{enumerate}
\end{thrm}

\begin{proof}
{$\mathbf{(2)\Rightarrow(1)}$} Suppose that the element $s$ is torsion. Let $\tilde{s} = t\bar\sigma$ denote a preimage of $s$ in $H^{0,2}(X)\simeq\tSha$. There exists a class $ l\in W_{\mathbb Q}$ such that $q(l,\sigma) = t$ (Lemma \ref{torsion in sha}). As $X$ is projective, there exists a class $\alpha\in H^{1,1}_{\mathbb Q}(X)$ such that $q(\alpha,\eta) = 1$ (Lemma \ref{characterization of projectivity}). The class $\alpha^s := \alpha - l$ satisfies $q(\alpha^s,\eta) = q(\alpha,\eta) = 1$ as $l$ is contained in $W_{\mathbb Q}\subset \eta^\perp$. We claim that $\alpha^s$ lies in $H^{1,1}_{\mathbb Q}(X^{\tilde s})$. Indeed, $\alpha^s$ is orthogonal to $H^{2,0}(X^{\tilde s})$ because
$$
q(\alpha^s,\sigma + t\eta) = q(\alpha - l,\sigma + t\eta) = q(\alpha,\sigma) - tq(l,\eta) = t-t = 0.
$$
Lemma \ref{characterization of projectivity} implies that $X^s$ is projective.

\hfill

{$\mathbf{(1)\Rightarrow(2)}$} Suppose that $X$ and $X^s$ are both projective. Let $\tilde s = t\bar\sigma$ denote a preimage of $s$ in $H^{0,2}(X)$ as before. There exist two rational classes $\alpha\in H^{1,1}_{\mathbb Q}(X)$ and $\alpha^s\in H^{1,1}_{\mathbb Q}(X^{\tilde s})$ such that $q(\alpha,\eta) = q(\alpha^s,\eta) = 1$ (Lemma \ref{characterization of projectivity}). Hence, the class $l:=\alpha - \alpha^s$ is a rational class orthogonal to $\eta$. This class satisfies
$$
q(l,\sigma) = q(l,\sigma + t\eta) = q(\alpha,\sigma + t\eta) = aq(\alpha,\eta) = t.
$$
We can not conclude directly that $[t\overline{\sigma}]$ is torsion because $l$ might not lie in $W_{\Q}$. Therefore, we need to adjust $l$.

Consider the subspace $W^\perp\subset H^2(X,\mathbb Q)$. It follows from Proposition \ref{W is big enough} that 
$$
\eta\in W^{\perp} \subset \left (\eta^{\perp} \cap H^{1,1}(X) \right).
$$
The BBF form $q$ has signature $(1, h^{1,1}(X)-1)$ on $H^{1,1}(X)$ and $\eta$ is isotropic with respect to this form. Therefore, the restriction of $q$ to $W^{\perp}$  is semi-negative definite with kernel generated by $\eta$.

 Let $U\subset W^\perp$ be a rational hyperplane in $W^\perp$ not containing $\eta$. The form $q|_U$ is negative definite, in particular, it is non-degenerate. Therefore there exists a unique rational vector $u\in U$ such that for every $v\in U$ the following holds
$$
q(l,v) = q(u,v).
$$
The vector $l-u$ is orthogonal to every vector in $W^\perp$. Hence $l-u$ is contained in $W_{\mathbb Q}$. Since $u\in H^{1,1}(X)$, we have $q(l-u,\sigma) = q(l,\sigma) = t$. Lemma \ref{characterization of projectivity} concludes the proof.
\end{proof}

\begin{rmk}
By \cite[Cor. 3.4]{soldatenkov2021moser} the set of $s\in\Sha^0$ such that $X^s$ is projective is non-empty.
\end{rmk}

\begin{rmk}
Theorem \ref{projectivity=torsion} was strengthened by the first author, see \cite[Theorem A]{abasheva2024shafarevich}
\end{rmk}

%\begin{cor}\label{torsor over torsion}
%Let $\pi \colon X \to B$ be a Lagrangian fibration on a hyperk\"ahler manifold. Assume that $X$ is not M-special. Then the set $$\mathcal{R}:= \{s \in \Sha^0 \ | \ X^s \textnormal{ is projective} \} \subset \Sha^0$$ is a non-empty torsor for the torsion subgroup of $\Sha^0$.
%\end{cor}
%\begin{proof}
%There exists a degenerate twistor deformation $X^s$ of $X$ which is projective \cite[Cor. 3.4]{SV}. As $X$ is assumed to be not M-special all the degenerate twistor deformations of $X$ are K\"ahler (Theorem \ref{Kahlerness}). Theorem \ref{projectivity=torsion} applied to $X^s$ concludes the proof.
%\end{proof}

%%%%%%%%%%%%%%%
%%%%%%%%%%%%%%

\section{Sections of Lagrangian fibrations}\label{obstruction}

\subsection{Obstruction for existence of a section}\label{obstruction construction}

We move on to study obstructions to existence of sections of Lagrangian fibrations. In this Section, we will always assume that $\pi\colon X\to B$ is a Lagrangian fibration with reduced irreducible fibers. 

In this case, the fibration $\pi \colon X \to B$ admits  a local section in a neighborhood of every point $b\in B$. Consider an open cover $B = \bigcup U_i$ and choose a collection of local sections $s_i \colon U_i \to \pi^{-1}(U_i)$. 

By \cite[Prop. 2.1 (iii)]{markushevich1996lagrangian} for every pair $i,j$ there exists a unique automorphism $\phi_{ij}\in Aut^0_{X/B}(U_{ij})$ such that $\phi_{ij}(s_i|_{U_{ij}}) = s_j|_{U_{ij}}$. The collection of automorphisms $\{\phi_{ij}\}$ satisfies the cocycle condition. Therefore, $\{\phi_{ij}\}$ defines a class 
$$\alpha(X, \pi)\in H^1(B, Aut^0_{X/B}) = \Sha.
$$
We denote this class by $\alpha(X)$ when the structure of the Lagrangian fibration on $X$ is clear.

\begin{lemma}\label{main property of alpha}
\begin{itemize}
    \item [(1)] The class $\alpha(X)$ does not depend on the choice of local sections $s_i\colon U_i\to \pi^{-1}(U_i)$.
    \item[(2)] For any element $s\in \Sha$ we have $\alpha(X^s) = \alpha(X) + s$.
    \item [(3)] The class $\alpha(X)$ vanishes if and only if the fibration $\pi\colon X\to B$ admits a section.
\end{itemize}
\end{lemma}

\begin{proof}
We will prove only the third statement. The proof of the first two ones follows the same lines. Suppose that $\alpha(X) = 0$. Then there exist automorphisms $\phi_i\in Aut^0_{X/B}(U_i)$ such that $\phi_{ij} = \psi^{-1}_j\psi_i$. The sections $\phi_i(s_i)$ coincide on intersections, so they define a global section of $\pi$. The converse implication is straightforward.
\end{proof}

Let $a(X)$ to be the image of $\alpha(X)$ in $\Sha/\Sha^0 = H^2(B,\Gamma)$.

\begin{cor}\label{main property of a}
Let $\pi\colon X\to B$ be a Lagrangian fibration with reduced irreducible fibers. Then the class $a(X)$ vanishes if and only if there exists a deformation $X^s$ of $X$ in the Shafarevich--Tate family that admits a holomorphic section. Moreover, in this case, the class of $s$ in $\Sha^0$ is uniquely defined.
\end{cor}
\begin{proof}
The class $a(X)\in H^2(B,\Gamma)$ vanishes if and only if $\alpha:=\alpha(X)$ lies in $\Sha^0$. The class $\alpha(X^{-\alpha})$ vanishes by Lemma \ref{main property of alpha} (2). Therefore 
$$
\pi^{-\alpha} \colon X^{-\alpha} \to B
$$
admits a holomorphic section. Conversely, if $\alpha \notin \Sha^0$, then for every deformation $X^s$ in the Shafarevich--Tate family we have $a(X^s) \neq 0$.
\end{proof}

It was proved in \cite[Thm. 3.5]{bogomolov2022sections} that a Lagrangian fibration admits a  \textit{smooth} section if and only if some of its degenerate twistor deformations admits a holomorphic section. Combined with Corollary \ref{main property of a} we get that $a(X)$ is indeed a complete topological obstruction for existence of a section on a Lagrangian fibration with reduced irreducible fibers.

For the rest of the paper we will be proving the following theorem.

\begin{thrm}\label{when a vanishes}
Let $\pi \colon X \to B$ be a Lagrangian fibration on a compact hyperk\"ahler manifold over a smooth base. Let $\Gamma$ be the kernel of the exponential map $\pi_*T_{X/B} \to Aut^0_{X/B}$. Assume that the following holds:
\begin{itemize}
\item the fibers of $\pi$ are reduced and irreducible;
\item $H^3(X, \Q) =0$;
\item $H^2(B, \Gamma)$ is torsion-free.
\end{itemize}
Then there exists a unique deformation $(X^s, \pi^s)$ of $(X, \pi)$ in the Shafarevich--Tate family such that $\pi^s \colon X^s \to B$ admits a holomorphic section.
\end{thrm}

Note that the condition $H^3(X, \Q)=0$ holds if $X$ is deformation equivalent to the Hilbert scheme of points on a K3 surface or to one of the exceptional O'Grady examples.

Unfortunately, we are not able to get rid of the condition on $H^2(B, \Gamma)$. However, we strongly believe, that the theorem should be true without this assumption. Therefore we pose the following conjecture.

\begin{conj}
Let $\pi\colon X\to B$ be a Lagrangian fibration. Assume that $b^3(X) = 0$. Then there exists a degenerate twistor deformation of $\pi$ admitting a holomorphic section.
\end{conj}

We believe that the assumption that fibers are reduced and irreducible is not too restrictive. %Unfortunately, we are not able to prove this directly, so we formulate it as conjectures.

\begin{conj}[\cite{bogomolov}]
\label{first conjecture}
Let $\pi\colon X\to B$ be a Lagrangian fibration on a compact hyperk\"ahler manifold. Then it has no multiple fibres.
\end{conj}
\begin{conj}
\label{second conjecture}
Let $\pi\colon X\to B$ be a Lagrangian fibration. Consider the space $\mathcal M_\eta$ of deformations of $X$ such that $\eta$ remains of type $(1,1)$. Then a very general deformation of $X$ in $\mathcal M_\eta$ is a Lagrangian fibration with irreducible fibers.
\end{conj}

Both conjectures hold for K3 surfaces (see \cite[Prop. 1.6 (ii)]{huybrechts2016lectures} for Conjecture \ref{first conjecture} and \cite[I.1.Thrm. 4.8]{friedman1994smooth} for Conjecture \ref{second conjecture}).

\subsection{Hard Lefschetz type theorems for higher direct images of $\mathbb Q_X$}\label{hard lefschetz}

Let $\pi \colon X \to B$ be a Lagrangian fibration. In this Subsection we  prove a version of the Hard Lefschetz theorem for the sheaf $R^1\pi_*\Q_X$ which will be used in the proof of the Theorem \ref{when a vanishes}. Throughout this Subsection we assume that $X$ is projective. Let $l\in H^{1,1}_\Q (X)$ be an ample class. Abusing notation, we will denote by the same letter the induced section of $R^2\pi_*\Q_X$.

For each $p\ge 0$, multiplication by $l$ induces a map
$$
    L\colon R^p\pi_*\Q_X \to R^{p+2}\pi_*\Q_X.
$$
We will refer to $L$ as the \textit{Lefschetz map}.

\begin{lemma}[Lefschetz decomposition for $R^2\pi_*\Q_X$]\label{Lefschetz decomposition}
Let $\pi\colon X\to B$ be a Lagrangian fibration with $X$ projective. Assume that all fibers of $\pi$ are reduced and irreducible. Then the sheaf $R^2\pi_*\Q_X$ decomposes as
$$
R^2\pi^*\Q_X = \Q_B \cdot l\oplus (R^2\pi_*\Q_X)_{prim}
$$
where $(R^2\pi_*\Q_X)_{prim}$ is the kernel of the map 
$$
L^{n-1}\colon R^2\pi_*\Q_X \to R^{2n}\pi_*\Q_X.
$$
\end{lemma}

\begin{proof}
Since the fibres are irreducible we have $R^{2n}\pi_*\Q_X \simeq \Q_B$. The restriction of $L^{n-1} \colon R^2\pi_*\Q_X \to R^{2n}\pi_*\Q_X$ on the subsheaf generated by $l$ is an isomorphism. Hence the claim.
\end{proof}

We move on to study the $(n-1)$-th power of the Lefschetz map on $R^1\pi_*\Q_X$, that is
\begin{equation}
\label{L on one}
L^{n-1}\colon R^1\pi_*\Q_X \to R^{2n-1}\pi_*\Q_X.
\end{equation}
First, note that multiplication by $l\in H^{1,1}_\Q(X)$ induces maps on $R^p\pi_*\Omega^q_X$ as well:
$$
    L\colon R^p\pi_*\Omega^q_X \to R^{p+1}\pi_*\Omega^{q+1}_X.
$$
for any $p,q = 0,\dots, n$. Abusing notation, we will denote these maps also by $L$. Lefschetz maps commute with the natural morphisms  $R^{\bullet}\pi_*\Q_X \to R^{\bullet}\pi_*\O_X$, the \textit{relative Hodge projections}. In particular, the following diagram is commutative.
$$
    \xymatrix{
    R^1\pi_*\Q_X \ar[r] \ar[d]_{L^{n-1}} & R^1\pi_*\mathcal O_X \ar[d]^{L^{n-1}}\\
    R^{2n-1}\pi_*\Q_X \ar[r] & R^n\pi_* \Omega^{n-1}_X
    }
$$
The horizontal arrows in this diagram are Hodge projections. The holomorphic symplectic form $\sigma$ on $X$ induces the isomorphism $\Omega^1_X\simeq T_X$. The composition of this isomorphism with the natural map $T_X\to\pi^* T_B$ gives us the map
$$
\theta\colon \Omega^1_X\to \pi^* T_B.
$$
For every pair of integers $p,q$ one can apply the functor $R^q\pi_*\Lambda^p(-)$ and get the following map of sheaves on $B$:
$$
R^q\pi_*(\Lambda^p\theta)\colon R^q\pi_*\Omega^p_X \to R^q\pi_*(\pi^*\Lambda^pT_B).
$$
The sheaf on the right is:
$$
R^q\pi_*(\pi^*\Lambda^p T_B) \simeq R^q\pi_*\mathcal O_X \otimes \Lambda^p T_B \simeq \Omega^q_B\otimes \Lambda^p T_B\simeq R^q\pi_*\mathcal O_X \otimes \left(R^p\pi_*\mathcal O_X\right)^*.
$$
The isomorphisms follow from the projection formula and Matsushita's theorem (Proposition \ref{Matsushita isomorphism}).

For each $p,q$ we have constructed a map
$$
f_{p,q}\colon R^q\pi_*\Omega^p_X \to R^q\pi_*\mathcal O_X\otimes \left(R^p\pi_*\mathcal O_X\right)^*.
$$
In particular, there are the following maps:
\begin{gather*}
    f_{0,1}\colon R^1\pi_*\mathcal O_X \to R^1\pi_*\mathcal O_X\\
    f_{1,1}\colon R^1\pi_*\Omega^1_X \to \mathcal{End} (R^1\pi_*\mathcal O_X)\\
    f_{n-1,n}\colon R^n\pi_*\Omega^{n-1}_X \to R^n\pi_*\mathcal O_X\otimes \left(R^{n-1}\pi_*\mathcal O_X\right)^* \simeq R^1\pi_*\mathcal O_X.
\end{gather*}
The map $f_{0,1}$ is the identity map by construction. It is easy to see that the image of  $l$ under $f_{1,1}$ is the identity operator.

\begin{lemma}
\label{decomposition coherent}
$f_{n-1,n}\circ L^{n-1} = \mathrm{id}_{R^1\pi_*\mathcal O_X}.$
\end{lemma}

\begin{proof}
The maps $f_{p,q}$ commute with the multiplication of forms in the sense that the following diagram is commutative for any $p_1,q_1,p_2,q_2 = 0,\dots,n$.
$$
    \xymatrix{
    R^{q_1}\pi_*\Omega^{p_1}_X\otimes R^{q_2}\pi_*\Omega^{p_2}_X \ar[r]\ar[d]^{f_{p_1,q_1}\otimes f_{p_2,q_2}} & R^{q_1+q_2}\pi_*\Omega^{p_1+p_2}_X \ar[d]^{f_{p_1+p_2,q_1+q_2}}\\
    R^{q_1}\pi_*\mathcal O_X\otimes \left(R^{p_1}\pi_*\mathcal O_X\right)^* \otimes     R^{q_2}\pi_*\mathcal O_X\otimes \left(R^{p_2}\pi_*\mathcal O_X\right)^*
     \ar[r] &   R^{q_1+q_2}\pi_*\mathcal O_X\otimes \left(R^{p_1+p_2}\pi_*\mathcal O_X\right)^*
    } 
$$
It follows that the following diagram is commutative
$$
    \xymatrix{
    R^1\pi_*\mathcal O_X \o \left(R^1\pi_* \Omega^1_X\right)^{\o\: n-1} \ar[r] \ar[d]^{\mathrm{id}\:\o (f_{1,1})^{\o\: n-1}} & R^n\pi_*\Omega^{n-1}_X \ar[d]^{f_{n-1,n}}\\
    R^1\pi_*\mathcal O_X\o \left(\mathcal{End}(R^1\pi_*\mathcal O_X)\right)^{\o\: n-1} \ar[r] & R^1\pi_*\mathcal O_X
    }
$$
Since $f_{1,1}(l) = \mathrm{id}_{R^1\pi_*\O_X}$, we obtain that for every local section $\alpha$ of $R^1\pi_*\O_X$
$$
f_{n-1,n}(\alpha\cdot l^{n-1}) = f_{n-1,n}\circ L^{n-1}(\alpha) = \alpha.
$$
\end{proof}

\begin{cor}
\label{Lefschetz}
Let $\pi\colon X\to B$ be a Lagrangian fibration on a projective hyperk\"ahler manifold $X$. Then there exists a sheaf $\mathcal N$ on $B$ such that the sheaf $R^{2n-1}\pi_*\Q_X$ decomposes into the direct sum
$$
R^{2n-1}\pi_*\mathbb Q_X \simeq R^1\pi_*\mathbb Q_X \oplus \mathcal N.
$$
The embedding of the first summand is given by the map (\ref{L on one}).
\end{cor}
\begin{proof}
The map (\ref{L on one}) is an isomorphism after the restriction to $B^\circ\subset B$ by the Hard Lefschetz theorem. Together with Lemma \ref{decomposition coherent} this implies that the map $f_{n-1,n}|_{B^\circ}$ sends $R^{2n-1}\pi_*\mathbb Q|_{B^\circ}$ isomorphically to $R^1\pi_*\mathbb Q|_{B^\circ}$. A local section of $R^1\pi_*\mathcal O_X$ whose restriction to $B^\circ$ lies in $R^1\pi_*\mathbb Q$ is necessarily a section of $R^1\pi_*\mathbb Q$ by the proof of Proposition \ref{Gamma is almost R^1pi_*Z}. Hence the map $f_{n-1,n}$ descends to a map
$$
f_{n-1, n}|_{R^{2n-1}\pi_*\Q_X}\colon R^{2n-1}\pi_*\mathbb Q\to R^1\pi_*\mathbb Q.
$$
This map satisfies the following property (Lemma \ref{decomposition coherent}):
$$
f_{n-1,n}|_{R^{2n-1}\pi_*\Q_X}\circ L^{n-1} = \mathrm{id}_{R^1\pi_*\Q_X}.
$$
The claim now follows.
\end{proof}

\begin{rmk}
It can be proven that the sheaf $\mathcal N$ from Proposition \ref{Lefschetz} is supported on a codimension two subset. The result follows from the description of a general singular fiber of a Lagrangian fibration given in \cite{hwang2009characteristic}. We do not know whether the sheaf $\mathcal N$ can be non-trivial.
\end{rmk}

%%%%%%%%%%%%%%%%%%%
%%%%%%%%%%%%%%%%%%%

\subsection{The discrete part of Shafarevich--Tate groups}\label{obstruction b_3}

Let $\Sha$ be the Shafarevich--Tate group of a Lagrangian fibration $\pi\colon X\to B$. Recall that the group $\Sha/\Sha^0$ of connected components of $\Sha$ is isomorphic to $H^2(B,\Gamma)$ (see the exact sequence (\ref{sha zero to sha})). We sometimes refer to $\Sha/\Sha^0$ as the discrete part of $\Sha$.

By Proposition \ref{Gamma is almost R^1pi_*Z}, there is an isomorphism
$$
H^2(B,\Gamma)\o_\mathbb Z \Q \simeq H^2(B, R^1\pi_*\Q_X).
$$
%Let us draw the second page of the Leray spectral sequence for the sheaf $\Q_X$.
%\begin{sseqdata}[name = sseqq, xscale = 2.4, classes = {draw = none}, cohomological Serre grading]
%\class["0"](-1,2)
%\class["0"](-1,1)
%\class["0"](-1,0)
%\class["H^0(R^2\pi_*\mathbb Q)"](0,2)
%\class["H^0(R^1\pi_*\mathbb Q)"](0,1)
%\class["\mathbb Z"](0,0)
%\class["H^1(R^1\pi_*\mathbb Q)"](1,1)
%\class["0"](1,0)
%\class["H^2(R^1\pi_*\mathbb Q)"](2,1)
%\class["\mathbb Q"](2,0)
%\class["0"](3,0)
%\end{sseqdata}
%\begin{sseqpage}[name = sseqq, page = 2]
%\classoptions[blue](2,1)
%\classoptions[blue](0,2)
%\d[blue]2(0,2)
%\end{sseqpage}
The natural map $E^{p,0}_2 = H^p(B, \Q) \to H^p(X, \Q)$ is given by the pullback map $\pi^*$. The pullback map on cohomology is injective for every surjective map of compact K\"ahler manifolds (see e.g. \cite[Lem. 7.28]{voisin2007hodge}). This implies that $E^{p,0}_2 = E^{p,0}_{\infty}$. Therefore, the differential $d_2\colon E_2^{2,1}\to E_2^{4,0}$ vanishes. All the higher differentials with the source in $E^{2,1}$ vanish because their targets are trivial groups. For the same reason, all the higher differentials $d_n, n>2$ with the source in $E^{0,2}$ vanish too. Since $E^{3,0}_2 = H^3(B,\Q)$ is trivial, there is an embedding of $E^{2,1}_\infty = E^{2,1}_2/\operatorname{im} (d_2)$ into $H^3(X,\Q)$. Moreover, we have the following exact sequence of $\Q$-vector spaces
\begin{equation}
\label{amazing exact sequence}
    H^2(X,\Q) \xrightarrow{r} H^0(B,R^2\pi_*\Q_X) \xrightarrow{d_2} H^2(B,R^1\pi_*\Q_X) \xrightarrow{e} H^3(X,\Q).
\end{equation}

Our goal is to prove the following theorem:

\begin{thrm}\label{differential vanishes}
Let $\pi\colon X\to B$ be a Lagrangian fibration with reduced irreducible fibers. Then the differential
$$
d_2\colon H^0(B,R^2\pi_*\Q_X)\to H^2(B,R^1\pi_*\Q_X)
$$
in the Leray spectral sequence of $\pi$ vanishes.
\end{thrm}

\begin{proof} Assume that $X$ is projective. Lemma \ref{Lefschetz decomposition} implies that
$$
H^0(B, R^2\pi_*\Q_X) = H^0(B,\Q)\oplus H^0(B,(R^2\pi_*\Q_X)_{prim})
$$
where $(R^2\pi_*\Q_X)_{prim}$ is the kernel of the map $L^{n-1}\colon R^2\pi_*\Q_X\to R^{2n}\pi_*\Q_X$. The summand $H^0(B,\Q)$ is generated by the image of the ample class $l$ of $X$ in $H^0(R^2\pi_*\Q_X)$. Hence $d_2|_{H^0(B,\Q)}$ vanishes.

We are left to prove that $d_2|_{H^0((R^2\pi_*\Q_X)_{prim})}$ vanishes. The differentials in the Leray spectral sequence commute with the Lefschetz maps. In particular, the following diagram is commutative.
$$
\xymatrix{
H^0((R^2\pi_*\Q_X)_{prim})\ar[r]^{d_2} \ar[d]^{L^{n-1}} & H^2(R^1\pi_*\Q_X)\ar[d]^{L^{n-1}}\\
H^0(R^{2n}\pi_*\Q_X)\ar[r]^{d_2} & H^2(R^{2n-1}\pi_*\Q_X)
}
$$
The vertical arrow on the left-hand side vanishes by the definition of the primitive part of $R^2\pi_*\Q_X$. The vertical map on the right-hand side is injective. Indeed, by Corollary \ref{Lefschetz} the map $L^{n-1}$ embeds $H^2(R^1\pi_*\Q_X)$ into $H^2(R^{2n-1}\pi_*\Q_X)$ as a direct summand. It follows that $d_2|_{H^0((R^2\pi_*\Q_X)_{prim})}$ must vanish. 

\hfill

{\bf Step 2:} In the case when $X$ is only assumed to be K\"ahler there exists a degenerate twistor deformation $\pi'\colon X'\to B$ of $X$ such that $X'$ is projective \cite[Cor. 3.4]{soldatenkov2021moser}. The statement of the theorem holds for $X'$. Since topologically the maps $\pi$ and $\pi'$ coincide, the statement of the theorem holds for $X$ as well. 
\end{proof}

\begin{cor}
\label{restriction is surjective}
In the setting of Theorem \ref{differential vanishes}, the following holds:
\begin{itemize}
    \item [(1)] The restriction map $r\colon H^2(X,\Q)\to H^0(B,R^2\pi_*\Q_X)$ is surjective.
    \item [(2)] The map $e\colon H^2(R^1\pi_*\Q_X)\to H^3(X,\Q)$ from the exact sequence (\ref{amazing exact sequence}) is injective.
\end{itemize}
\end{cor}
\begin{proof}
Follows from Theorem \ref{differential vanishes} and the exact sequence (\ref{amazing exact sequence}).
\end{proof}

\noindent {\bf Proof of Theorem \ref{when a vanishes}:} It follows from Corollary \ref{restriction is surjective} (2) and Proposition \ref{Gamma is almost R^1pi_*Z} that there exists an embedding
\[
(\Sha/\Sha^0)\o_{\Z} \Q \hookrightarrow H^3(X, \Q).
\]
In particular, if $b_3(X)$ vanishes, $\Sha/\Sha^0$ is finite. Since $\Sha/\Sha^0 \simeq H^2(B, \Gamma)$ this finishes the proof.\qed

\begin{rmk}
If $\pi\colon X\to B$ is an elliptic fibration on a K3 surface, then $\Sha/\Sha^0$ is trivial \cite[I.1, Lemma 5.1]{friedman1994smooth}.
\end{rmk}

%%%%%%%%%%%%%%%%%
%%%%%%%%%%%%%%%%

\bibliographystyle{alpha}
\bibliography{main.bib}

\begin{multicols}{2}
\footnotesize
\noindent {\sc {\small Anna Abasheva} \\
Columbia University\\
Department of Mathematics, \\
2990 Broadway, New York, NY, USA}\\
{\tt aa4643(at)columbia.edu}

\columnbreak

\noindent {\sc {\small Vasily Rogov}\\
Max-Planck-Institut f\"ur Mathematik in den Naturwissenschaften. Inselstraße 22, 04103 Leipzig, Germany.}\\
{\tt  vasirog (at) gmail.com }

\end{multicols}

\end{document}